\documentclass[reqno]{amsart}

\usepackage{amsfonts}
\usepackage{amssymb}
\usepackage{amsmath}
\usepackage{enumitem}
\usepackage{xcolor}

\setenumerate{label={\rm (\alph{*})}}
\usepackage{bbm,euscript,mathrsfs}
\usepackage{paper_diening}
\usepackage{graphicx}
\usepackage{hyperref}
\usepackage[top=1in, bottom=1.25in, left=1.10in, right=1.10in]{geometry}

\allowdisplaybreaks

\usepackage{tikz-cd}
\usetikzlibrary{lindenmayersystems}
\usetikzlibrary{decorations.pathreplacing}

\DeclareMathOperator{\Bog}{Bog}

\DeclareMathOperator{\Div}{div}

\newcommand{\R}{\mathbb R}
\newcommand{\N}{\mathbb N}
\newcommand{\dd}{\mathrm d}
\newcommand{\dt}{\,\mathrm{d} t}

\newcommand{\dx}{\,\mathrm{d}x}
\newcommand{\dz}{\,\mathrm{d}z}

\newcommand{\dy}{\,\mathrm{d}y}

\newcommand{\dxt}{\,\mathrm{d}x\,\mathrm{d}t}
\newcommand{\ds}{\,\mathrm{d}\sigma}
\newcommand{\dxs}{\,\mathrm{d}x\,\mathrm{d}\sigma}

\newcommand{\mt}{\Omega}

\newtheorem{theorem}{Theorem}[section]
\newtheorem{lemma}[theorem]{Lemma}

\newtheorem{remark}[theorem]{Remark}

\theoremstyle{definition}
\newtheorem{definition}[theorem]{Definition}

\begin{document}

\title[Boundary regularity Navier--Stokes]
{Partial boundary regularity for the Navier--Stokes equations in irregular domains}

\author{Dominic Breit}
\address{Department of Mathematics, Heriot-Watt University, Edinburgh, EH14 4AS, United Kingdom}
\email{d.breit@hw.ac.uk}
\address{Institute of Mathematics, TU Clausthal, Erzstra\ss e 1, 38678 Clausthal-Zellerfeld, Germany}
\email{dominic.breit@tu-clausthal.de}



\subjclass[2010]{76Nxx; 76N10; 35Q30 ; 35Q84; 82D60}

\date{\today}


\keywords{}

\begin{abstract}
We prove partial regularity of suitable weak solutions to the Navier--Stokes equations at the boundary in irregular domains. In particular, we provide a criterion which yields continuity of the velocity field in a boundary point and obtain solutions which are continuous in a.a. boundary boundary point (their existence is a consequence of a new maximal regularity result for the Stokes equations in domains with minimal regularity).
We suppose that we have a Lipschitz boundary with locally small Lipschitz constant which belongs to the fractional Sobolev space $W^{2-1/p,p}$ 
for some $p>\frac{15}{4}$. The same result was previously only known under the much stronger assumption of a $C^2$-boundary.
\end{abstract}

\maketitle

\section{Introduction}
We consider the motion of a viscous incompressible fluid in a physical body $\Omega$ -- a bounded Lipschitz domain in $\R^3$ during the interval $\mathcal I=(0,T)$. 
The motion of the fluid is governed by the Navier--Stokes equations
 \begin{align} \varrho\big(\partial_t \bfu+(\bfu\cdot\nabla)\bfu\big)&=\mu\Delta\bfu-\nabla \pi+\bff,\quad \Div\bfu=0,&\label{1}
 \end{align}
 in $\mathcal I\times\Omega$, where $\bfu:\mathcal I\times\Omega\rightarrow\R^3$ is the velocity field and $\pi:\mathcal I\times\Omega\rightarrow\R$ the pressure function. The quantity $\bff:\mathcal I\times\Omega\rightarrow \R^3$ is an external forcing, $\varrho$ is the density and $\mu$ the viscosity -- two positive constants which will be set to 1 in the following for simplicity.
 
The existence of weak solutions to \eqref{1} has been established in the 1930's by Leray \cite{Ler}.
 The regularity of solutions to \eqref{1} is an outstanding open problem which has been attracting mathematicians for decades -- still we are far away from a complete understanding (though some remarkable recent progress on the non-uniqueness of weak solutions has been made based on the method
of convex integration, cf. \cite{Col,BV,BCV}). 
 The state of the art today is partial regularity. This means that the velocity field is locally bounded/H\"older continuous outside a negligible set of the space-time cylinder (further regularity properties inside this set can be deduced) with measure zero. Such an analysis has been initiated in a series of papers by Sheffer, see \cite{Sh1}--\cite{Sh4}. A further milestone is the work by Caffarelli-Kohn-Nirenberg in \cite{CKN} on suitable weak solutions. These solutions satisfy a form of the energy inequality which is localised in space-time (hereafter called local energy inequality) which reads as
 \begin{equation}\label{energylocal0a}
\begin{split}
\int_ {\mt}\frac{1}{2}&\zeta\big| \bfu(t)\big|^2\dx+\int_0^t\int_ {\mt}\zeta|\nabla \bfu |^2\dx\ds\\& \leq\int_0^t\int_{\mt}\frac{1}{2}\Big(| \bfu|^2(\partial_t\zeta+\Delta\zeta)+\big(|\bfu|^2+2\pi\big)\bfu\cdot\nabla\zeta\Big)\dx\ds+\int_0^t\int_{\mt}\zeta\bff\cdot\bfu\dxs
\end{split}
\end{equation}
for any $\zeta\in C_c^\infty(\mathcal I\times\Omega)$ with $\zeta\geq0$.
  This is a piece of information which has to be included into the definition of a solution as it is otherwise lost in the construction procedure. It is not known (and maybe not even expected) if any weak solution satisfies a local energy inequality (in fact, the same problem appears for the global energy inequality too). As demonstrated in \cite{CKN}, an analysis of the local regularity properties of suitable weak solutions is possible. In particular, a criterion for a solution to \eqref{1} is provided, which yields boundedness of the velocity field in a given point in space-time.
  Some further improvements and simplifications can be found in \cite{CL,LaSe,L}.
 
 The regularity of solutions to \eqref{1} at the boundary seems to be less understood. A first results has been achieved by Sheffer in \cite{Sh4} and a more systematic analysis was started by Seregin in \cite{Se1}--\cite{Se3}. Still all these results consider the case of a flat boundary. A flat boundary is the easiest case to consider and one expects that the same results also apply for curved boundaries provided they are sufficiently smooth. Indeed, a corresponding theory for non-flat boundaries of class $C^2$ has been obtained in \cite{SeShSo}. Many applications naturally lead, however, to a boundary of significantly less regularity.
 This is particularly motivated by problems from fluid-structure interaction, where the boundary is described by the displacement of an elastic structure. The latter is the solution to a partial differential equation on its own and hence only of limited regularity (we comment further on this in Section \ref{sec:fsi}).
 
 We aim to prove partial regularity of solutions to \eqref{1} at the boundary under minimal assumptions on the regularity of the boundary. Our analysis is based on the concept of boundary suitable weak solutions as in \cite{Se1,SeShSo}.
  They satisfy a local energy inequality near a certain part of the boundary (that is, \eqref{energylocal0a} for cut-off functions supported near a part of the boundary), see Definition \ref{def:weakSolution}. So far, even their existence was not known for boundaries with regularity below $C^2$. In order to improve this
  we prove a new result on the maximal $L^r_tL^p_x$-theory for the unsteady Stokes system under minimal assumptions
  on the boundary regularity, cf. Theorem \ref{thm:stokesunsteady}.
  The main assumption is that $\Omega$ is a Lipschitz domain with locally small Lipschitz constant and that the local coordinates (we make this concept precise in Section \ref{sec:para}.) belong to the class of Sobolev multipliers on $W^{2-1/p,p}$ -- the trace space of $W^{2,p}$ to which the velocity field belongs.
  The $L^r_tL^p_x$-theory just described provides a parabolic counterpart of the recent results on the steady Stokes system from \cite{Br}
  and yields the existence of boundary suitable weak solutions to the Navier--Stokes system \eqref{1} in irregular domains,
cf. Theorem \ref{thm:existence}.

Eventually, we prove a criterion for boundary suitable weak solutions to \eqref{1} which implies continuity of the velocity field in a boundary point (see Theorem \ref{thm:main}) and hence obtain solutions which are continuous
  in almost any boundary boundary point (see Theorem \ref{thm:main'}). Our main assumption is that the boundary coordinates belong to the class of Sobolev multipliers on 
\begin{align}\label{ass:main}  
W^{2-\frac{1}{p},p}(\R^2)\quad\text{for some}\quad p>\tfrac{15}{4}
\end{align}
 with sufficiently small norm. This class includes Lipschitz boundaries with small Lipschitz constant belonging to the class $W^{\sigma,p}(\R^2)$ for $\sigma>2-1/p$, see Remark \ref{rem:boundary}.
  
Our approach uses a flattening of the boundary by means of a transformation $\bfPhi\in W^{2,p}(\R^3)$ which is an extension of the function $\varphi$ describing the boundary locally. This leads to some kind of perturbed Navier--Stokes equations (or perturbed Stokes equations if the convective term is neglected).
 We provide a regularity theory for the perturbed Stokes system under minimal assumptions on the coefficients resulting from the flattening (this is similar to that of the Stokes system in irregular domains mentioned above). This is used in the partial regularity proof via the blow-up technique.
 There are various stages in the proof of the blow up lemma (see Lemma \ref{Lemma}), which require restrictions on the regularity of $\bfPhi$ (or that of $\varphi$). The most restrictive one is related to the decay of the pressure $\tilde{\mathfrak q}$ of some perturbed Stokes system: 
 We have to show that
\begin{align}\label{eq:1401}
\tau^{3}\bigg(\dashint_{\mathcal I_\tau}\dashint_{\mathcal B^+_\tau}  |\tilde{\mathfrak{q}}  - (\tilde{\mathfrak q})_{\mathcal B^+_{\tau}} |^{5/3} \dz\ds\bigg)^{\frac{9}{5}}\lesssim\tau^{2\alpha}
\end{align}
for some $\alpha>0$ to arrive at a contradiction. Here $\mathcal Q^+_\tau=\mathcal I_\tau\times\mathcal B_\tau^+$ denotes a parabolic half-cylinder centered at the flat boundary with radius $\tau>0$. Estimate \eqref{eq:1401} can be proved by Poincar\'e's inequality if $\nabla\tilde{\mathfrak{q}}$ belongs to the space $L^{5/3}_tL^p_x$ with $p>\frac{15}{4}$. 
The term in \eqref{eq:1401} results from our unusual choice of excess functional related to the integrability $L_t^{5/3}L_x^{5/3}$ for the pressure. In order to fine-tune the assumptions on $p$ in \eqref{ass:main} (see Section \ref{sec:opt} for some discussion) we must choose the time-integrability of the pressure as large as possible, where the upper limit is $5/3$ due to the integrability of the convective term.

 Details on the $L^{r}_tL^p_x$-theory for the perturbed Stokes system  can be found in Lemma \ref{lem:31} (which also implies a useful Caccioppoli-type inequality in Lemma \ref{lem:32}). The latter also yields an estimate for the velocity field $\overline\bfu$ of the perturbed Stokes system in the same space which implies its continuity (we expect this to be true also under weaker assumptions, see the discussion in Section \ref{sec:opt}). This is needed similarly to \eqref{eq:1401} for the decay of the perturbed velocity field.
Finally, the transformation of the local energy inequality from the original to the flat geometry requires some assumptions on the boundary, though weaker than those already mentioned (we refer to the estimates in \eqref{eq:Im2}).
At first glance it seems that one needs $\bfPhi\in W^{2,\infty}$ to control this term. A more careful analysis reveals, however, that it is sufficient if $\bfPhi$ belongs to the Sobolev multiplier class on $W^{2,3/2}$. 
Although it remains unclear at this stage if \eqref {ass:main} is optimal for partial boundary regularity, we believe that it will be very difficult to relax it.  
We comment further on this in Section \ref{sec:opt}.

\section{Preliminaries and results}
\subsection{Conventions}
We write $f\lesssim g$ for two non-negative quantities $f$ and $g$ if there is a $c>0$ such that $f\leq\,c g$. Here $c$ is a generic constant which does not depend on the crucial quantities and can change from line to line. If necessary we specify particular dependencies. We write $f\approx g$ if $f\lesssim g$ and $g\lesssim f$.
We do not distinguish in the notation for the function spaces between scalar- and vector-valued functions. However, vector-valued functions will usually be denoted in bold case.

\subsection{Classical function spaces}
Let $\mathcal O\subset\R^m$, $m\geq 1$, be open.
Function spaces of continuous or $\alpha$-H\"older continuous functions, $\alpha\in(0,1)$,
 are denoted by $C(\overline{\mathcal O})$ or $C^{0,\alpha}(\overline{\mathcal O})$ respectively. Similarly, we write $C^1(\overline{\mathcal O})$ and $C^{1,\alpha}(\overline{\mathcal O})$.
We denote as usual by $L^p(\mathcal O)$ and $W^{k,p}(\mathcal O)$ for $p\in[1,\infty]$ and $k\in\mathbb N$ Lebesgue and Sobolev spaces over $\mathcal O$. For a bounded domain $\mathcal O$ the space $L^p_\perp(\mathcal O)$ denotes the subspace of  $L^p(\mathcal O)$ of functions with zero mean, that is $(f)_{\mathcal O}:=\dashint_{\mathcal O}f\dx:=\mathcal L^m(\mathcal O)^{-1}\int_{\mathcal O}f\dx=0$.
 We denote by $W^{k,p}_0(\mathcal O)$ the closure of the smooth and compactly supported functions in $W^{k,p}(\mathcal O)$. This coincides with the functions vanishing $\mathcal H^{m-1}$ -a.e. on $\partial\mathcal O$ provided $\partial\mathcal O$ is sufficiently regular. 
 We also denote by $W^{-k,p}(\mathcal O)$ the dual of $W^{k,p}_0(\mathcal O)$.
  Finally, we consider the subspace
$W^{1,p}_{0,\Div}(\mathcal O)$ of divergence-free vector fields which is defined accordingly. 
We will use the shorthand notations $L^p_x$ and $W^{k,p}_x$ in the case of $3$-dimensional domains and   
$L^p_y$ and $W^{k,p}_y$ for $2$-dimensional sets.

For a separable Banach space $(X,\|\cdot\|_X)$ we denote by $L^p(0,T;X)$ the set of (Bochner-) measurable functions $u:(0,T)\rightarrow X$ such that the mapping $t\mapsto \|u(t)\|_{X}\in L^p(0,T)$. 
The set $C([0,T];X)$ denotes the space of functions $u:[0,T]\rightarrow X$ which are continuous with respect to the norm topology on $(X,\|\cdot\|_X)$.
 The space $W^{1,p}(0,T;X)$ consists of those functions from $L^p(0,T;X)$ for which the distributional time derivative belongs to $L^p(0,T;X)$ as well. 
We use the shorthand $L^p_tX$ for $L^p(0,T;X)$. For instance, we write $L^p_tW^{1,p}_x$ for $L^p(0,T;W^{1,p}(\mathcal O))$. Similarly, $W^{k,p}_tX$ stands for $W^{k,p}(0,T;X)$.

The space $C^{\alpha,\beta}([0,T]\times \overline{\mathcal O})$ with $\alpha,\beta\in(0,1]$ denotes the set of functions being $\alpha$-H\"older continuous
in $t\in[0,T]$ and $\beta$-H\"older continuous in $x\in\overline{\mathcal O}$.

\subsection{Fractional differentiability and Sobolev mulitpliers}
\label{sec:SM}
For $p\in[1,\infty)$ the fractional Sobolev space (Sobolev-Slobodeckij space) with differentiability $s>0$ with $s\notin\mathbb N$ will be denoted by $W^{s,p}(\mathcal O)$. For $s>0$ we write $s=\lfloor s\rfloor+\lbrace s\rbrace$ with $\lfloor s\rfloor\in\N_0$ and $\lbrace s\rbrace\in(0,1)$.
 We denote by $W^{s,p}_0(\mathcal O)$ the closure of the smooth and compactly supported functions in $W^{1,p}(\mathcal O)$. For $s>\frac{1}{p}$ this coincides with the functions vanishing $\mathcal H^{m-1}$ -a.e. on $\partial\mathcal O$ provided $\partial\mathcal O$ is regular enough. We also denote by $W^{-s,p}(\mathcal O)$ for $s>0$ the dual of $W^{s,p}_0(\mathcal O)$. Similar to the case of unbroken differentiabilities above we use the shorthand notations $W^{s,p}_x$  and $W^{s,p}_y$.. 
We will denote by $B^s_{p,q}(\R^m)$ the standard Besov spaces on $\R^m$ with differentiability $s>0$, integrability $p\in[1,\infty]$ and fine index $q\in[1,\infty]$. They can be defined (for instance) via Littlewood-Paley decomposition leading to the norm $\|\cdot\|_{B^s_{p,q}(\R^m)}$. 
 We refer to \cite{RuSi} and \cite{Tr,Tr2} for an extensive picture. 
 The Besov spaces $B^s_{p,q}(\mathcal O)$ for a bounded domain $\mathcal O\subset\R^m$ are defined as the restriction of functions from $B^s_{p,q}(\R^m)$, that is
 \begin{align*}
 B^s_{p,q}(\mathcal O)&:=\{f|_{\mathcal O}:\,f\in B^s_{p,q}(\R^m)\},\\
 \|g\|_{B^s_{p,q}(\mathcal O)}&:=\inf\{ \|f\|_{B^s_{p,q}(\R^m)}:\,f|_{\mathcal O}=g\}.
 \end{align*}
 If $s\notin\mathbb N$ and $p\in(1,\infty)$ we have $B^s_{p,p}(\mathcal O)=W^{s,p}(\mathcal O)$.
 
In accordance with \cite[Chapter 14]{MaSh} the Sobolev multiplier norm  is given by
\begin{align}\label{eq:SoMo}
\|\varphi\|_{\mathcal M^{s,p}(\mathcal O)}:=\sup_{\bfv:\,\|\bfv\|_{W^{s-1,p}(\mathcal O)}=1}\|\nabla\varphi\cdot\bfv\|_{W^{s-1,p}(\mathcal O)},
\end{align}
where $p\in[1,\infty]$ and $s\geq1$.
The space $\mathcal M^{s,p}(\mathcal O)$ of Sobolev multipliers is defined as those objects for which the $\mathcal M^{s,p}(\mathcal O)$-norm is finite. For $\delta>0$ we denote by  $\mathcal M^{s,p}(\mathcal O)(\delta)$ the subset of functions from
 $\mathcal M^{s,p}(\mathcal O)$ with $\mathcal M^{s,p}(\mathcal O)$-norm not exceeding $\delta$.
By mathematical induction with respect to $s$ one can prove for Lipschitz-continuous functions $\varphi$ that membership to $\mathcal M^{s,p}(\mathcal O)$  in the sense of \eqref{eq:SoMo} implies that
\begin{align}\label{eq:SoMo'}
\sup_{w:\,\|w\|_{W^{s,p}(\mathcal O)}=1}\|\varphi \,w\|_{W^{s,p}(\mathcal O)}<\infty.
\end{align}
The quantity \eqref{eq:SoMo'} also serves as customary definition of the Sobolev multiplier norm in the literature but \eqref{eq:SoMo} is more suitable for our purposes.

Let us finally collect some some useful properties of Sobolev multipliers.
By \cite[Corollary 14.6.2]{MaSh} we have
\begin{align}\label{eq:MSa}
\|\phi\|_{\mathcal M^{s,p}(\R^{m})}\lesssim\|\nabla\phi\|_{L^{\infty}(\R^m)},
\end{align}
provided that one of the following conditions holds:
\begin{itemize}
\item $p(s-1)<m$ and $\phi\in B^{s}_{\varrho,p}(\R^{m})$ with $\varrho\in\big[\frac{pm}{p(s-1)-1},\infty\big]$;
\item $p(s-1)=m$ and $\phi\in B^{s}_{\varrho,p}(\R^m)$ with $\varrho\in(p,\infty]$.
\end{itemize}
Note that the hidden constant in \eqref{eq:MSa} depends on the $B^{s}_{\varrho,p}(\R^{m})$-norm of $\phi$.
By \cite[Corollary 4.3.8]{MaSh} it holds
\begin{align}\label{eq:MSb}
\|\phi\|_{\mathcal M^{s,p}(\R^{m})}\approx
\|\nabla\phi\|_{W^{s-1,p}(\R^{m})} 
\end{align}
for $p(s-1)>m$. 
 Finally, we note the following rule about the composition with Sobolev multipliers which is a consequence of \cite[Lemma 9.4.1]{MaSh}. For open sets $\mathcal O_1,\mathcal O_2\subset\R^m$, $u\in W^{s,p}(\mathcal O_2)$ and a Lipschitz continuous function $\bfphi:\mathcal O_1\rightarrow\mathcal O_2$ with $\bfphi\in \mathcal M^{s,p}(\mathcal O_1)$ and Lipschitz continuous inverse $\bfphi^{-1}:\mathcal O_2\rightarrow\mathcal O_1$ we have
\begin{align}\label{lem:9.4.1}
\|u\circ\bfphi\|_{W^{s,p}(\mathcal O_1)}\lesssim \|u\|_{W^{s,p}(\mathcal O_2)}
\end{align}
with constant depending on $\bfphi$. 

\subsection{Parametrisation of domains}\label{sec:para}
In this section we present the necessary framework to parametrise
the boundary of the underlying domain $\Omega\subset\R^3$ by local maps of a certain regularity. This yields, in particular, a rigorous definition of a  $\mathcal M^{s,p}$-boundary. We follow the presentation from \cite{Br}.

We assume that $\partial\Omega$ can be covered by a finite
number of open sets $\mathcal U^1,\dots,\mathcal U^\ell$ for some $\ell\in\mathbb N$, such that
the following holds. For each $j\in\{1,\dots,\ell\}$ there is a reference point
$y^j\in\R^3$ and a local coordinate system $\{e^j_1,e^j_2,e_3^j\}$ (which we assume
to be orthonormal and set $\mathcal Q_j=(e_1^j|e_2^j |e_3^j)\in\mathbb R^{3\times 3}$), a function
$\varphi_j:\mathbb R^{2}\rightarrow\mathbb R$
and $r_j>0$
with the following properties:
\begin{enumerate}[label={\bf (A\arabic{*})}]
\item\label{A1} There is $h_j>0$ such that
$$\mathcal U^j=\{x=\mathcal Q_jz+y^j\in\mathbb R^3:\,z=(z',z_3)\in\R^3,\,|z'|<r_j,\,
|z_3-\varphi_j(z')|<h_j\}.$$
\item\label{A2} For $x\in\mathcal U^j$ we have with $z=\mathcal Q_j^t(x-y^j)$
\begin{itemize}
\item $x\in\partial\Omega$ if and only if $z_3=\varphi_j(z')$;
\item $x\in\Omega$ if and only if $0<z_3-\varphi_j(z')<h_j$;
\item $x\notin\Omega$ if and only if $0>z_3-\varphi_j(z')>-h_j$.
\end{itemize}
\item\label{A3} We have that
$$\partial\Omega\subset \bigcup_{j=1}^\ell\mathcal U^j.$$
\end{enumerate}
In other words, for any $x_0\in\partial\Omega$ there is a neighborhood $U$ of $x_0$ and a function $\phi:\mathbb R^{2}\rightarrow\mathbb R$ such that after translation and rotation\footnote{By translation via $y_j$ and rotation via $\mathcal Q_j$ we can assume that $x_0=0$ and that the outer normal at~$x_0$ is pointing in the negative $x_3$-direction.}
 \begin{align}\label{eq:3009}
 U \cap \Omega = U \cap G,\quad G = \set{(x',x_3)\in \R^3 \,:\, x' \in \R^{2}, x_3 > \phi(x')}
 \end{align}
 The regularity of $\partial\Omega$ will be described by means of local coordinates as just described.
 \begin{definition}\label{def:besovboundary}
 Let $\Omega\subset\R^3$ be a bounded domain, $s\geq 1$ and $p\in[1,\infty]$. We say that $\partial\Omega$ belongs to the class $\mathcal M^{s,p}$ if there is $\ell\in\mathbb N$ and functions $\varphi_1,\dots,\varphi_\ell\in \mathcal M^{s,p}(\R^2)$  satisfying \ref{A1}--\ref{A3}.
 \end{definition}
 Clearly, we  can define similarly a $\mathcal M^{s,p}(\delta)$-boundary for some $\delta>0$ by requiring that $\varphi_1,\dots,\varphi_\ell\in \mathcal M^{s,p}(\R^2)(\delta)$.
Analogous definitions apply for various other function spaces such as $B^s_{\varrho,p}$ for $s>0$ and $\varrho,p\in[1,\infty]$ or $C^{1,\alpha}$ for $\alpha\in(0,1]$. Of particular importance for us is also a Lipschitz boundary, where $\varphi_1,\dots,\varphi_\ell\in W^{1,\infty}(\mathbb R^{2})$. We say that the Lipschitz constant of $\partial\Omega$, denoted by $\mathrm{Lip}(\partial\Omega)$, is (smaller or) equal to some number $L>0$ provided the Lipschitz constants of $\varphi_1,\dots,\varphi_\ell$ are not exceeding $L$. 
\begin{remark}\label{rem:boundary}
It follows from \eqref{eq:MSa} and \eqref{eq:MSb} that $\partial\Omega\in \mathcal M^{s,p}(\delta)$
provided $\Omega$ is a Lipschitz domain with sufficiently small Lipschitz constant and $\partial\Omega\in B^\theta_{\varrho, p}$ for $\theta>s$ and
\begin{align}\label{eq:SMp}
\varrho\geq p\quad\text{if}\quad p(s-1)\geq 3,\quad \varrho\geq \tfrac{2p}{p(s-1)-1}\quad\text{if}\quad p(s-1)< 3,
\end{align}
The Lipschitz constant can be made sufficiently small, for instance, if $\partial\Omega\in C^{1,\alpha}$ for some $\alpha>0$.
\end{remark}

In order to describe the behaviour of functions defined in $\Omega$ close to the boundary we need to extend the functions $\varphi_1,\dots,\varphi_\ell$  from \ref{A1}--\ref{A3} to the half space
$\mathbb H := \set{\xi = (\xi',\xi_3)\,:\, \xi_3 > 0}$. Hence we are confronted with the task of extending a function~$\phi\,:\, \R^{2}\to \R$ to a mapping $\Phi\,:\, \mathbb H \to \R^3$ that maps the 0-neighborhood in~$\mathbb H$ to the $x_0$-neighborhood in~$\Omega$. The mapping $(\xi',0) \mapsto (\xi',\phi(\xi'))$ locally maps the boundary of~$\mathbb H$ to the one of~$\partial \Omega$. We extend this mapping using the extension operator of Maz'ya and Shaposhnikova~\cite[Section 9.4.3]{MaSh}. Let $\zeta \in C^\infty_c(B_1(0'))$ with $\zeta \geq 0$ and $\int_{\R^{2}} \zeta(x')\dx'=1$. Let $\zeta_t(x') := t^{-2} \zeta(x'/t)$ denote the induced family of mollifiers. We define the extension operator 
\begin{align*}
  (\mathcal{T}\phi)(\xi',\xi_3)=\int_{\R^{2}} \zeta_{\xi_3}(\xi'-y')\phi(y')\dy',\quad (\xi',\xi_3) \in \mathbb H,
\end{align*}
where~$\phi:\R^2\to \R$ is a Lipschitz function with Lipschitz constant~$L$.
Then the estimate 
\begin{align}\label{est:ext}
  \norm{\nabla (\mathcal{T} \phi)}_{W^{s,p}(\setR^{3})}\le c\norm{\nabla \phi}_{W^{s-\frac 1 p,p}(\setR^{2})}
\end{align}
follows from~\cite[Theorem 8.7.2]{MaSh}. Moreover, \cite[Theorem 8.7.1]{MaSh} yields
\begin{align}\label{eq:MS}
\|\mathcal T\phi\|_{\mathcal M^{s,p}(\mathbb H)}\lesssim \|\phi\|_{\mathcal M^{s-\frac{1}{p},p}(\R^{2})}.
\end{align}
It is shown in \cite[Lemma 9.4.5]{MaSh} that (for sufficiently large~$N$, i.e., $N \geq c(\zeta) L+1$) the mapping
\begin{align*}
  \alpha_{z'}(z_3) \mapsto N\,z_3+(\mathcal{T} \phi)(z',z_3)
\end{align*}
is for every $z' \in \setR^{2}$ one to one and the inverse is Lipschitz with gradient
bounded by $(N-L)^{-1}$.
Now, we define the mapping~$\bfPhi\,:\, \mathbb H \to \R^3$ as a rescaled version of the latter one by setting
\begin{align}\label{eq:Phi}
  \bfPhi(\xi',\xi_3)
  &:=
    \big(\xi',
    \alpha_{\xi_3}(\xi')\big) = 
    \big(\xi',
    \,\xi_3 + (\mathcal{T} \phi)(\xi',\xi_3/N)\big).
\end{align}
Thus, $\bfPhi$ is one-to-one (for sufficiently large~$N=N(L)$) and we can define its inverse $\bfPsi := \bfPhi^{-1}$.
 The Jacobi matrix of the mapping $\bfPhi$ satisfies
\begin{align}\label{J}
  J = \nabla \bfPhi = 
  \begin{pmatrix}
    \mathbb I_{2\times 2}&0
    \\
    \partial_{\xi'} (\mathcal{T}  \phi)& 1+ 1/N\partial_{\xi_n}\mathcal{T}  \phi
  \end{pmatrix}.
\end{align}
Since 
$\abs{\partial_{\xi_3}\mathcal{T}  \phi} \leq L$, we have \begin{align}\label{eq:detJ}\frac{1}{2} < 1-L/N \leq \abs{\det(J)} \leq 1+L/N\leq 2\end{align}
using that $N$ is large compared to~$L$. Finally, we note the implication
\begin{align} \label{eq:SMPhiPsi}
\bfPhi\in\mathcal M^{s,p}(\mathbb H)\,\,\Rightarrow \,\,\bfPsi\in \mathcal M^{s,p}(\mathbb H),
\end{align}
which holds, for instance, if $\bfPhi$ is Lipschitz continuous, cf. \cite[Lemma 9.4.2]{MaSh}.

\begin{remark}\label{rem:cover}
\begin{enumerate}
\item
Since the cover $\mathcal U^1,\dots,\mathcal U^\ell$ is open it is possible to find a number $\mathfrak R>0$ (depending on the cover) such that the following holds: for every $x\in\partial\Omega$ there is $j\in\{1,\dots,\ell\}$ such that $x\in \mathcal U^j$ and $\mathrm{dist}(x,\mathbb R^3\setminus \mathcal U^j)\geq \mathfrak R$. 
\item Similarly, by possibly decreasing $\mathfrak R$, we have the following: there is $\delta>0$ (depending on the cover) such that for any $x\in\Omega$ with $\mathrm{dist}(x,\partial\Omega)\leq \delta$
there is $j\in\{1,\dots,\ell\}$ such that $x\in \mathcal U^j$ and $\mathrm{dist}(x,\mathbb R^3\setminus \mathcal U^j)\geq \mathfrak R$.
\end{enumerate}
\end{remark}

\subsection{The main results}\label{sec:mainresults}
We start with a definition of boundary suitable weak solutions adapting the notation from \cite{SeShSo}. These solutions satisfy a local form of the energy inequality in the neighborhood of boundary points. 
For that purpose we fix two numbers $r_\ast\in(1,2)$ and $s_\ast\in(1,\frac{3}{2})$ such that $\frac{1}{r_\ast}+\frac{3}{2s_\ast}\geq 2$. The choice comes from the fact that the convective term $(\nabla\bfu)\bfu$ of a weak solution to \eqref{1} belongs to $L^{r_\ast}_tL^{s_\ast}_x$. Later on we will choose $r_\ast=5/3$ and, accordingly, $s_\ast=15/14$.
\begin{definition}[Boundary suitable weak solution] \label{def:weakSolution}
Let $\Omega\subset\R^3$ be a bounded Lipschitz domain and $\Gamma\subset\partial\Omega$ relatively open.
Let $(\bff, \bfu_0)$ be a dataset such that
\begin{equation}
\begin{aligned}
\label{dataset}
&\bff \in L^{r_\ast}\big(\mathcal I; L^{s_\ast}(\Omega)\big),\quad 
\bfu_0\in W^{2,s_\ast}\cap W^{1,2}_{0,\mathrm{\Div}}(\Omega).
\end{aligned}
\end{equation} 
We call the triple
$(\bfu,\pi,\varphi)$
a boundary suitable weak solution to the Navier--Stokes system \eqref{1} hear $\Gamma$ with data $(\bff, \bfu_0)$ provided that the following holds:
\begin{itemize}
\item[(a)] The velocity field $\bfu$ satisfies
\begin{align*}
 \bfu \in L^\infty \big(\mathcal I; L^2(\Omega) \big)\cap  L^2 \big(\mathcal I; W^{1,2}_{0,\Div}(\Omega) \big) \cap L^{r_\ast}(\mathcal I;W^{2,s_\ast}(\Omega))\cap W^{1,r_\ast}(\mathcal I;L^{s_\ast}(\Omega)).
\end{align*}
\item[(b)] The pressure $\pi$ satisfies
$$\pi\in L^{r_\ast}(\mathcal I;W^{1,s_\ast}_\perp(\Omega)).$$
\item[c)] 
 We have\footnote{after translation via $y_j$ and rotation via $\mathcal Q_j$, cf. Section \ref{sec:para}.}
 \begin{align*}
 \Gamma \subset \partial\Omega \cap G,\quad G = \set{(x',x_3)\in \R^n \,:\, x' \in \R^{2}, x_3 = \varphi(x')}
 \end{align*}
 for some Lipschitz function $\varphi:\R^2\rightarrow\R$ satisfying 
\begin{align}\label{eq:varphi}
\varphi(0)=0,\quad \nabla\varphi(0)=0.
\end{align}
\item[(d)] We have
\begin{align*}
\partial_t \bfu+(\bfu\cdot\nabla)\bfu&=\Delta\bfu-\nabla \pi+\bff,\quad \Div\bfu=0,\quad\bfu|_{\partial\Omega}=0,\quad \bfu(0,\cdot)=\bfu_0,
\end{align*}
a.a. in $\mathcal I\times\Omega$.
\item[(e)] for any $\zeta\in C_c^\infty(\mathcal I\times\R^3)$ with $\zeta\geq0$ and $\mathrm{spt}(\zeta) \subset \mathcal I\times(\R^3\setminus \Gamma)$ the local energy inequality 
\begin{equation}\label{energylocal0}
\begin{split}
\int_ {\mt}\frac{1}{2}&\zeta\big| \bfu(t)\big|^2\dx+\int_0^t\int_ {\mt}\zeta|\nabla \bfu |^2\dx\ds\\& \leq\int_0^t\int_{\mt}\frac{1}{2}\Big(| \bfu|^2(\partial_t\zeta+\Delta\zeta)+\big(|\bfu|^2+2\pi\big)\bfu\cdot\nabla\zeta\Big)\dxs+\int_0^t\int_{\mt}\zeta\bff\cdot\bfu\dxs.
\end{split}
\end{equation}
holds.
\end{itemize}
\end{definition}
The next theorem shows that, under suitable assumptions on $\partial\Omega$, there is a solution to \eqref{1} which is a  boundary suitable weak solution around every boundary point.
\begin{theorem}\label{thm:existence}
Suppose that $\Omega\subset\R^3$ is a bounded Lipschitz domain such that $\mathrm{Lip}(\partial\Omega)\leq\delta$ and $\partial\Omega\in \mathcal M^{2-1/s_\ast,s_\ast}(\delta)$ for some sufficiently small $\delta$. Then there is a solution $(\bfu,\pi)$ to the Navier--Stokes equations \eqref{1} with the following property: For every point $x_0\in\partial\Omega$ there is a neighborhood $\mathcal U(x_0)$ such that $(\bfu,\pi)$ generates a boundary suitable weak solution to the Navier--Stokes system \eqref{1} 
near $\mathcal U(x_0)\cap\partial\Omega$ in the sense of Definition \ref{def:weakSolution}.
\end{theorem}
\begin{proof}
Applying a standard regularisation procedure (by convolution with a mollifying kernel) to the functions $\varphi_1,\dots,\varphi_\ell$ from \ref{A1}--\ref{A3} in the parametrisation of $\partial\Omega$ we obtain a smooth boundary. Classically, the solution to the corresponding Stokes system is smooth.
 Such a procedure is standard and has been applied, for instance, in \cite{cm-1,CiMa}. It is possible to do this in a way that the original domain is included in the regularised domain to which we extend the functions $\bff$ and $\bfu_0$ by means of an extension operator.
 The regularisation applied to the $\varphi_j's$ converges on all Besov spaces with $p<\infty$. As shown in \cite[Lemma 4.3.3.]{MaSh}
it does not expand the $\mathcal M^{s,p}(\R^{2})$-norm, which is sufficient. For a smooth domain the statement of Theorem \ref{thm:existence} is well-known (it can, for instance, be proved along the lines of \cite{CKN}). We obtain a sequence of functions $(\bfu_m,\pi_m)$ which satisfy the Navier--Stokes equations as well as the local energy inequality \eqref{energylocal0}. Clearly, they can be constructed to also satisfy the global energy inequality. Hence our sequence $(\bfu_m)$ is bounded in the energy space
 \begin{align*}
 L^\infty(\mathcal I;L^2(\Omega))\cap L^2(\mathcal I;W^{1,2}_{0,\Div}(\Omega)).
 \end{align*}
 As a consequence, we can bound the convective term $(\nabla\bfu_m)\bfu_m$ in $L^{r_\ast}(I;L^{s_\ast}(\Omega))$. Now we come to the crucial point: By the maximal regularity theory for the Stokes system from Theorem \ref{thm:stokesunsteady} below we have
 \begin{align}\label{eq:2505}
 \begin{aligned}
\|\partial_t\bfu_m\|_{L^{r_\ast}(\mathcal I;L^{s_\ast}(\Omega))}&+\|\bfu_m\|_{L^{r_\ast}(\mathcal I;W^{2,s_\ast}(\Omega))}
+\|\pi_m\|_{L^{r_\ast}(\mathcal I;W^{1,s_\ast}(\Omega))}\\&\lesssim \|\bff\|_{L^{r_\ast}(\mathcal I;L^{s_\ast}(\Omega)}+\|(\nabla\bfu_m)\bfu_m\|_{L^{r_\ast}(\mathcal I;L^{s_\ast}(\Omega)}+\|\bfu_0\|_{W^{2,s_\ast}(\Omega)}
\end{aligned}
\end{align}
uniformly in $m$. 
With \eqref{eq:2505} at hand, we obtain (after passing to a subsequence) limit objects with the claimed regularity and can pass to the limit in the momentum equation and local energy inequality.
\end{proof}

\begin{theorem}\label{thm:main}
Let $\Omega\subset\R^3$ be a bounded Lipschitz domain and $\Gamma\subset\partial\Omega$ relatively open. Suppose that $\bff\in L^p(\mathcal I;L^p(\Omega))$ for some $p>\frac{15}{4}$.
There is a number $\varepsilon_{0}>0$ such that the following holds. Let $(\bfu,\pi,\varphi)$ be a boundary suitable weak solution to the Navier--Stokes system \eqref{1} near $\Gamma$ in the sense of Definition \ref{def:weakSolution}, where $\varphi\in\mathcal M^{2-1/p,p}(\R^2)(\delta)$ with $\|\varphi\|_{W^{1,\infty}_y}\leq\delta$ and sufficiently small $\delta>0$.
 Let $x_0\in\Gamma$ and $t_0\in I$ such that 
\begin{align}
r^{-2}\int_{t_0-r^2}^{t_0+r^2}\int_{\Omega\cap \mathcal B_r(x_0)}|\bfu|^3\dxt+\bigg(r^{-5/3}\int_{t_0-r^2}^{t_0+r^2}\int_{\Omega\cap \mathcal B_r(x_0)}|\pi|^{5/3}\dx\dt\bigg)^{\frac{9}{5}}<\varepsilon_0
\end{align}
for some $r\leq \mathrm{dis}(x_0,\R^3\setminus\Gamma)$.
Then we have $\bfu\in C^{0,\alpha}(\overline{\mathcal U}(t_0,x_0))$ for some $\alpha>0$ and a neighborhood $\mathcal U(t_0,x_0)$ of $(t_0,x_0)$.
\end{theorem}
%
%
Denoting by $\mathcal H^s_{\mathrm{para}}$ the $s$-dimensional parabolic Hausdorff measure (and using its definition based on covering with parabolic cubes) it is standard to deduce the following result concerning the size of the singular set from Theorem \ref{thm:main}.
\begin{theorem}\label{thm:main'}
Suppose that $\Omega\subset\R^3$ is a bounded Lipschitz domain such that $\mathrm{Lip}(\partial\Omega)\leq\delta$ and $\partial\Omega\in \mathcal M^{2-1/p,p}(\delta)$ for some $p>\frac{15}{4}$ and sufficiently small $\delta$. Suppose that $\bff\in L^p(\mathcal I;L^p(\Omega))$.
Then there is a solution $(\bfu,\pi)$ to the Navier--Stokes equations \eqref{1} and a closed set $\Sigma\subset I\times \partial\Omega$ with $\mathcal H^{5/3}_{\mathrm{para}}(\Sigma)=0$ such that for any $(t_0,x_0)\in I\times \partial\Omega \setminus\Sigma$ we have $\bfu\in C^{0,\alpha}(\overline{\mathcal U}(t_0,x_0))$ for some $\alpha>0$ and a neighborhood $\mathcal U(t_0,x_0)$ of $(t_0,x_0)$.
\end{theorem}
Our result in Theorem \ref{thm:main'} is in terms of the size of the singular set weaker than the result in \cite{SeShSo} for more regular domains. It is shown there that the dimension of the singular set is one rather than $5/3$. We comment on this gap in more detail in Section \ref{sec:setsize}. It is unlcear if this is an intrinsic feature of irregular domains or a drawback of our method.

\section{The Stokes system in irregular domains}
\label{sec:stokesunsteady}
In this section we consider the unsteady Stokes system
\begin{align}\label{eq:Stokesunstay}
\partial_t\bfu=\Delta \bfu-\nabla\pi+\bff,\quad\Div\bfu=0,\quad\bfu|_{I\times\partial{\Omega}}=\bfu_{\partial},\quad \bfu(0,\cdot)=\bfu_0,
\end{align}
in a domain ${\Omega}\subset\R^3$ with unit normal $\bfn$. The result given in the following theorem is a maximal regularity estimate for the solution in terms of the right-hand side under minimal assumption on the regularity of $\partial{\Omega}$ (see Remark \ref{eq:SMp} for the connecton between Sobolev multipliers and Besov spaces).
\begin{theorem}\label{thm:stokesunsteady}
Let $p,r\in(1,\infty)$ and suppose that ${\Omega}$ is a Lipschitz domain with local Lipschitz constant $\delta$ belonging to the class $\mathcal M^{2-1/p,p}(\R^2)(\delta)$ for some sufficiently small $\delta$, $\bff\in L^r(\mathcal I;L^{p}({\Omega}))$ and $\bfu_{0}\in W^{2,p}({\Omega})\cap L^{2}_{\Div}({\Omega})$ with $\mathrm{tr}\,\bfu_0=\bfu_\partial$, where $\bfu_{\partial}\in W^{2-1/p,p}(\partial{\Omega})$ with $\int_{\partial{\Omega}}\bfu_\partial\cdot\bfn\,\dd\mathcal H^{2}=0$. Then there is a unique solution to \eqref{eq:Stokesunstay} and we have
\begin{align}\label{eq:mainpara}
\begin{aligned}
\|\partial_t\bfu\|_{L^r(\mathcal I;L^{p}({\Omega}))}&+\|\bfu\|_{L^r(\mathcal I;W^{2,p}({\Omega}))}
+\|\pi\|_{L^r(\mathcal I;W^{1,p}({\Omega}))}\\&\lesssim\|\bff\|_{L^r(\mathcal I;L^{p}({\Omega}))}+\|\bfu_{0}\|_{W^{2,p}({\Omega})}+\|\bfu_{\partial}\|_{W^{2-1/p,p}(\partial{\Omega})}.
\end{aligned}
\end{align}
\end{theorem}
\begin{remark}
As shown in \cite[Chapter 14]{MaSh} the assumptions on $\partial\Omega$ in Theorem \ref{thm:stokesunsteady} are sharp already for the Laplace equation.
\end{remark}
\begin{remark}
Theorem \ref{thm:stokesunsteady} provides a parabolic counterpart of the result in the steady Stokes system from \cite[Theorem 3.1]{Br} (with the same assumptions on $\partial\Omega$).
A direct adaption of the ideas used in the proof of the steady analogue from \cite[Theorem 3.1]{Br} is not straightforward due to the appearance of the time-derivative in the divergence-correction.
Hence we follow instead the classical approach from \cite{So} (see also the presentation in \cite[Appendix 2]{Brb}).
The idea is to first solve the problem in the flat geometry
and built the solution of the original problem by concatenating the (transformed) solutions. This is somewhat the opposite way compared to \cite[Theorem 3.1]{Br} and leads to various lower order error terms. They can be controlled for small times and we obtain a global-in-time solution by gluing local solutions together.
\end{remark}
\begin{remark}\label{rem:div}
Since $\mathcal O$ is assumed to be a Lipschitz domain the Bogovskii operator has the usual properties, cf. \eqref{eq:bog} below, and we can extend the result of Theorem \ref{thm:stokesunsteady} to the Stokes problem with a given divergence $h\in L^r(\mathcal I;W^{1,p}({\Omega}))\cap W^{1,r}(\mathcal I;W^{-1,p}(\mathcal O))$. In this case the additional terms
$$\|h\|_{L^r(\mathcal I;W^{1,p}(\mathcal O))},\quad \|h\|_{W^{1,r}(\mathcal I;W^{-1,p}(\mathcal O))},$$
appear on the right-hand side of \eqref{eq:mainpara}.
\end{remark}
\begin{proof}
By use of a standard extension operator we can assume that $\bfu_\partial=0$. Otherwise, we can solve the homogeneous problem with solution $\tilde{\bfu}$ and set $$\bfu:=\tilde\bfu+\mathcal E_{{\Omega}}\bfu_\partial-\Bog_{\Omega}(\Div \mathcal E_{{\Omega}}\bfu_\partial)$$ where $$\mathcal E_{{\Omega}}:W^{2-1/p,p}(\partial{\Omega})\rightarrow W^{2,p}({\Omega})$$
is a continuous linear extension operator and $\Bog_{\Omega}$ the Bogovskii-operator. The latter solves the divergence equation (with respect to homogeneous boundary conditions on $\partial{\Omega}$) and satisfies 
\begin{align}\label{eq:bog}
\Bog_{\Omega}\Div :W^{2,p}\cap \bigg\{\bfw:\,\int_{\partial{\Omega}}\bfw\cdot\bfn\,\dd\mathcal H^{2}=0\bigg\}\rightarrow W^{2,p}_0({\Omega})
\end{align} for all $p\in (1,\infty)$, see \cite{Ga}[Section III.3].

We want to invert the operator
\begin{align*}
\mathscr L&:\mathscr Y_{r,p}\rightarrow L^r(\mathcal I; L^p_{\Div}(\mathcal O)),\quad
\bfv\mapsto\mathcal P_p\big(\partial_t\bfv-\Delta\bfv\big),
\end{align*}
where the space $\mathscr Y_{r,p}$ is given by
\begin{align*}
\mathscr Y_{r,p}:= L^r(\mathcal I;W^{1,p}_{0,\Div}\cap W^{2,p}({\Omega}))\cap W^{1,r}(\mathcal I;L^{p}({\Omega}))\cap\set{\bfv(0,\cdot)=0}
\end{align*}
and $\mathcal P_p$ is the Helmholtz projection from $L^p({\Omega})$ onto $L^p_{\Div}({\Omega})$. 
The Helmholtz-projection $\mathcal P_p \bfu$ of a function $\bfu\in L^p({\Omega})$ is defined as $\mathcal P_p \bfu:=\bfu-\nabla h$, where $h$ is the solution to the Neumann-problem
\begin{align*}
\begin{cases}\Delta h=\Div \bfu\quad\text{in}\quad {\Omega},\\
\bfn\cdot(\nabla h-\bfu)=0\quad\text{on}\quad\partial {\Omega}.
\end{cases}
\end{align*}
We will try to find an operator $\mathscr R:L^r(\mathcal I; L^p_{\Div}(\Omega))\rightarrow\mathscr Y_{r,p}$ such that
\begin{align}\label{eq:IS}
\mathscr L\circ\mathscr R=\mathrm{id}+\mathscr T,
\end{align}
where the operator-norm of $\mathscr T$ is strictly smaller than 1. This implies that the range of $\mathscr L\circ\mathscr R$ (which then equals to $L^r(\mathcal I;L^p_{\Div}(\Omega))$) is contained in the range of $\mathscr L$. Hence $\mathscr L$ is onto.

By assumption there is $\ell\in\mathbb N$ and functions $\varphi_1,\dots,\varphi_\ell\in\mathcal M^{2-1/p,p}(\delta)(\mathbb R^{2})$ satisfying \ref{A1}--\ref{A3}.
We clearly find an open set $\mathcal U^0\Subset{\Omega}$ such that ${\Omega}\subset \cup_{j=0}^\ell \mathcal U^j$. Finally, we consider a decomposition of unity $(\xi_j)_{j=0}^\ell$ with respect to the covering
$\mathcal U^0,\dots,\mathcal U^\ell$ of ${\Omega}$. 
For $j\in\{1,\dots,\ell\}$ we consider the extension $\bfPhi_j$ of $\varphi_j$ given by \eqref{est:ext} with inverse $\bfPsi_j$.
Denoting by $\mathscr V_jx=\mathcal Q_j^\top(x-y_j)$ (with the translation $y_j$ and the rotation $\mathcal Q_j$ used in \ref{A1}--\ref{A3}, cf. Section \ref{sec:para})
we define the operators\footnote{Note that by means of a standard extension operator we extend functions in \eqref{eq:operators}--\eqref{eq:stokeshalf} to the whole space or half space when necessary.}
\begin{align}\label{eq:operators}
\mathscr R_0\bff&:=\xi_0 \bfU_0+\sum_{j=1}^\ell\xi_j \bfU_j \circ\bfPsi_j\circ\mathscr V_j ,\quad
\mathscr P\bff:=\sum_{j=1}^\ell\xi_j\mathfrak q_j \circ\bfPsi_j\circ\mathscr V_j.
\end{align}
Here the functions $(\bfU_0,\mathfrak q_0)$ and $(\bfU_j,\mathfrak q_j)$ for $j\in\{1,\dots,\ell\}$ are the solutions to the Stokes problem on the whole space and the half space with data $\bff$  respectively (transformed if necessary), that is, we have
\begin{align}\label{eq:stokeswhole}
\partial_t\bfU_0=\Delta \bfU_0-\nabla\mathfrak q_0+\bff,\quad\Div\bfU_0=0,\quad\bfU_0|_{I\times\partial\mathbb H}=0,\quad \bfU_0(0,\cdot)=0,
\end{align}
and 
\begin{align}\label{eq:stokeshalf}
\partial_t\bfU_j=\Delta \bfU_j-\nabla\mathfrak q_j+\bff\circ\mathscr V_j^{-1}\circ \bfPhi_j,\quad\Div\bfU_j=0,\quad\bfU_j|_{I\times\partial\mathbb H}=0,\quad \bfU_j(0,\cdot)=0.
\end{align}
We have 
\begin{align}\label{est:stokeswhole}
\begin{aligned}
\int_{\mathcal I}\Big(\|\partial_t\bfU_0\|_{L^{p}_x}^r+\|\nabla^2\bfU_0\|^r_{L^{p}_x}+\|\nabla\mathfrak q_0\|_{L^{p}_x}^r\Big)\dt&\lesssim \int_{\mathcal I}\|\bff\|_{L^{p}_x}^r\dt,\\
\int_{\mathcal I}\Big(\|\partial_t\bfU_0\|_{W^{-1,p}_x}^r+\|\bfU_0\|_{W^{1,p}_x}^r+\|\mathfrak q_0\|_{L^{p}_x}^r\Big)\dt&\lesssim T^{r/2}\int_{\mathcal I}\|\bff\|_{L^{p}_x}^r\dt,
\end{aligned}
\end{align}
and for $j=1,\dots,\ell$
\begin{align}\label{est:stokeshalf}
\begin{aligned}
\int_{\mathcal I}\Big(\|\partial_t\bfU_j\|_{L^{p}_x}^r+\|\nabla^2\bfU_j\|_{L^{p}_x}^r+\|\nabla\mathfrak q_j\|_{L^{r}_x}^r\Big)\dt&\lesssim\int_{\mathcal I}\|\bff\circ\mathscr V_j\circ \bfPhi_j\|_{L^{p}_x}^r\dt,\\
\int_{\mathcal I}\Big(\|\partial_t\bfU_j\|_{W^{-1,p}_x}^r+\|\bfU_j\|_{W^{1,p}_x}^r+\|\mathfrak q_j\|_{L^{p}_x}^r\Big)\dt&\lesssim T^{r/2}\int_{\mathcal I}\|\bff\circ\mathscr V_j\circ \bfPhi_j\|_{L^{p}_x}^r\dt,
\end{aligned}
\end{align}
uniformly in $T$. Note that estimates \eqref{est:stokeswhole}$_2$ and \eqref{est:stokeshalf}$_2$ only hold locally in space (that is, in balls $B\subset\R^n$ with a constant depending on the radius). This does not cause any problems due to the localisation functions appearing in \eqref{eq:operators}.  
 Since Lipschitz continuity of $\varphi_j$ implies that of $\bfPhi_j$, cf. \eqref{J},
 we can control the right-hand sides in \eqref{est:stokeshalf} by $\|\bff\|_{L^r_tL^{p}_x}^r$.
Estimates \eqref{est:stokeswhole} and \eqref{est:stokeshalf} are classical in the case $r=p$, see \cite[Theorems 3.1 \& 3.2]{So}.
For the case of arbitrary exponents $p$ and $r$ we refer to \cite{HS} and the references therein.
Note that the $T$-dependence in \eqref{est:stokeswhole}$_2$ and \eqref{est:stokeshalf}$_2$ follows by simple scaling argument.

The divergence of $\mathscr R_0\bff$  as defined in \eqref{eq:operators} is in general not zero.
This will be corrected by setting
\begin{align*}
\mathscr R\bff=\mathscr R_0\bff+\mathscr R_1\bff,\quad \mathscr R_1\bff=-\Bog_{{\Omega}}\Div\mathscr R_0\bff
\end{align*}
with the Bogovskii-operator $\Bog_{\Omega}$, cf. \eqref{eq:bog}.
Now we clearly have $\mathscr R\bff\in \mathscr Y_{r,p}$ and the aim is to establish (\ref{eq:IS}). Transforming $\bfU_j$ and $\mathfrak q_j$ back to ${\Omega}$, that is, setting $\bfV_j=\bfU_j\circ\bfPsi_j\circ\mathscr V_j$ and $\mathfrak Q_j=\mathfrak q_j\circ\bfPsi_j\circ\mathscr V_j $, we obtain
\begin{align}\label{eq:Stokesback}
\begin{aligned}
\partial_t\bfV_j=&\Delta\bfV_j-\nabla\mathfrak Q_j+(1-\mathrm{det}(\nabla\bfPsi_j))\partial_t\bfV_j-\Div\big((\mathbb I_{3\times 3}-\bfA_j)\nabla\bfV_j)-\Div((\mathbf{B}_j-\mathbb I_{3\times 3})\mathfrak Q_j)+\bff,\\
&\Div\bfV_j=(\mathbb I_{3\times 3}-\mathbf{B}_j)^\top:\nabla\bfV_j,\quad\bfV_j|_{\partial{\Omega}\cap \mathcal U^j}=0,\quad \bfV_j(0,\cdot)=0,
\end{aligned}
\end{align}
where $\bfA_j:=\mathrm{det}(\nabla\bfPsi_j)\nabla\bfPhi_j^\top\circ\bfPsi_j\nabla\bfPhi_j\circ\bfPsi_j$ and $\mathbf{B}_j:=\mathrm{det}(\nabla\bfPsi_j)\nabla\bfPhi_j\circ\bfPsi_j$.
 There holds
\begin{align}
\partial_t\mathscr R\bff&-\Delta\mathscr R\bff+\nabla \mathscr P\bff=\bff+\mathscr S\bff+(\partial_t-\Delta)\mathscr R_1\bff,\label{eq:4.10}\\
\mathscr S\bff&=-\nabla\bfV_0\nabla\xi_0-\Div\big(\nabla\xi_0\otimes\bfV_0\big)
-\sum_{j=1}^\ell\nabla\bfV_j \nabla\xi_j\nonumber\\&-\sum_{j=1}^\ell\Div\big(\nabla\xi_j\otimes\bfV_j\big)+\sum_{j=1}^\ell\nabla\xi_j \mathfrak Q_j
-\sum_{j=1}^\ell\xi_j\Div((\mathbf{B}_j-\mathbb I_{3\times 3})\mathfrak Q_j)\label{eq:4.11}\\
&-\sum_{j=1}^\ell\xi_j\Div\big((\mathbb I_{3\times 3}-\bfA_j)\nabla\bfV_j)+\sum_{j=1}^\ell\xi_j(1-\mathrm{det}(\nabla\bfPsi_j))\partial_t\bfV_j=:\sum_{i=1}^8 \mathscr S_i\bff.\nonumber
\end{align}
From (\ref{eq:4.10}) it follows
\begin{align*}
\mathscr L\mathscr R\bff=\bff+\mathcal P_p\mathscr S\bff+\mathcal P_p(\partial_t-\Delta)\mathscr R_1\bff,
\end{align*}
i.e., (\ref{eq:IS}) with $\mathscr T=\mathcal P_p\mathscr S+\mathcal P_p(\partial_t-\Delta)\mathscr R_1$. 
We claim that
\begin{align}\label{eq:0401}
\sum_{i=1}^5\int_{\mathcal I}\|\mathscr S_i\bff\|^r_{L^{p}_x}\dt\leq\,\delta(T)\int_{\mathcal I}\|\bff\|_{L^{p}_x}^r\dt
\end{align}
with $\delta(T)\rightarrow0$ for $T\rightarrow0$.\footnote{Note that this will also depend on $\ell$ which we have to choose sufficiently large to obtain $\max_j\|\varphi_j\|_{W^{1,\infty}_y}\approx \mathrm{Lip}(\partial{\Omega})$.} Estimate
\eqref{eq:0401} follows from estimates
\eqref{est:stokeswhole} and \eqref{est:stokeshalf}.
We can translate \eqref{est:stokeshalf} into an estimate
for $\bfV_j$ and $\mathfrak Q_j$
via
\begin{align}\label{eq:2901}
\begin{aligned}
\|\bfV_j\|_{W^{\sigma,p}_x}
\lesssim \|\bfU_j\|_{W^{\sigma,p}_x}
\end{aligned}
\end{align}
for $\sigma\in\{1,2\}$ and similarly
\begin{align}\label{eq:2901'}
\begin{aligned}
\|\mathfrak Q_j\|_{W^{\sigma-1,p}_x}
\lesssim \|\mathfrak q_j\|_{W^{\sigma-1,p}_x}
\end{aligned}
\end{align}
recalling estimates \eqref{eq:detJ} and \eqref{lem:9.4.1}. Note that by our assumptions on $\varphi_j$ and \eqref{eq:Phi} we have
 $\bfPhi_j\in \mathcal M^{2,p}(\mathbb H)$ and thus $\bfPsi_j\in \mathcal M^{2,p}(\mathbb H)$ by \eqref{eq:SMPhiPsi}. Combining the previous arguments proves \eqref{eq:0401}.

Now we are concerned with $\mathscr S_6$ and $\mathscr S_7$
obtaining
by \eqref{J}, \eqref{lem:9.4.1} and the definitions of $\bfA_j$ and $\bfPhi_j$
\begin{align*}
\|\xi_j\Div&\big((\mathbb I_{3\times 3}-\bfA_j)\nabla\bfV_j)\|_{L^{p}(\mathbb H)}\\&\lesssim \sup_{\|\bfw\|_{W^{1,p}_x}\leq 1}\|(\mathbb I_{3\times 3}-\bfA_j)\bfw\|_{W^{1,p}(\mathbb H)}\|\nabla\bfV_j\|_{W^{1,p}_x}\\
&\lesssim \sup_{\|\bfw\|_{W^{1,p}(\mathbb H)}\leq 1}\|(1-\mathrm{det}(\nabla\bfPsi_j))\bfw\|_{W^{1,p}(\mathbb H)}\|\bfV_j\|_{W^{2,p}_x}\\
&+ \sup_{\|\bfw\|_{W^{1,p}_x}\leq 1}\|\mathrm{det}(\nabla\bfPsi_j)(\mathbb I_{3\times 3}-\nabla\bfPhi_j^\top\circ\bfPsi_j)\bfw\|_{W^{1,p}(\mathbb H)}\|\bfV_j\|_{W^{2,p}_x}\\
&+ \sup_{\|\bfw\|_{W^{1,p}_x}\leq 1}\|\mathrm{det}(\nabla\bfPsi_j)\nabla\bfPhi_j^\top\circ\bfPsi_j(\mathbb I_{3\times 3}-\nabla\bfPhi_j\circ\bfPsi_j)\bfw\|_{W^{1,p}(\mathbb H)}\|\bfV_j\|_{W^{2,p}_x}\\
&\lesssim \|\mathcal T\phi_j\|_{\mathcal M^{2,p}(\mathbb H)}\|\bfv\|_{W^{2,p}_x}\\&+ \|\bfPsi_j\|_{\mathcal M^{2,p}(\mathbb H)}^3\sup_{\|\bfw\|_{W^{1,p}_x}\leq 1}\|(\mathbb I_{3\times 3}-\nabla\bfPhi_j\circ\bfPsi_j)\bfw\|_{W^{1,p}(\mathbb H)}\|\bfV_j\|_{W^{2,p}_x}\\
&+ \|\bfPsi_j\|_{\mathcal M^{2,p}(\mathbb H)}^3\|\bfPhi_j\|_{\mathcal M^{2,p}(\mathbb H)}\sup_{\|\bfw\|_{W^{1,p}_x}\leq 1}\|(\mathbb I_{3\times 3}-\nabla\bfPhi_j\circ\bfPsi_j)\bfw\|_{W^{1,p}(\mathbb H)}\|\bfV_j\|_{W^{2,p}_x}\\
&\lesssim \Big(\|\mathcal T\phi_j\|_{\mathcal M^{2,p}(\mathbb H)}+ \sup_{\|\bfw\|_{W^{1,p}_x}\leq 1}\|(\mathbb I_{3\times 3}-\nabla\bfPhi_j\circ\bfPsi_j)\bfw\|_{W^{1,p}(\mathbb H)}\Big)\|\bfV_j\|_{W^{2,p}_x},
\end{align*}
where 
\begin{align*}
\sup_{\|\bfw\|_{W^{1,p}_x}\leq 1}&\|(\mathbb I_{3\times 3}-\nabla\bfPhi_j\circ\bfPsi_j)\bfw\|_{W^{1,p}(\mathbb H)}\\
&\lesssim \|\mathcal T\phi_j\circ\bfPsi_j\|_{\mathcal M^{2,p}(\mathbb H)}\lesssim \|\mathcal T\phi_j\|_{\mathcal M^{2,p}(\mathbb H)}.
\end{align*}
So we finally have
\begin{align*}
\|\xi_j\Div\big((\mathbb I_{3\times 3}-\bfA_j)\nabla\bfV_j)\|_{L^{p}(\mathbb H)}&\lesssim \|\mathcal T\phi_j\|_{\mathcal M^{2,p}(\mathbb H)}\|\bfV_j\|_{W^{2,p}(\mathbb H)}
\end{align*}
and, similarly, 
\begin{align*}
\|\xi_j\Div((\mathbf{B}_j-\mathbb I_{3\times 3})\mathfrak Q_j)\|_{L^{p}(\mathbb H)}&\lesssim \sup_{\|\bfw\|_{W^{1,p}(\mathbb H)}\leq 1} \|(\mathbf{B}_j-\mathbb I_{3\times 3})\bfw\|_{W^{1,p}(\mathbb H)}\|\mathfrak Q_j\|_{W^{1,p}(\mathbb H)}\\
&\lesssim \|\mathcal T\phi_j\|_{\mathcal M(W^{2,p}(\mathbb H))}\|\mathfrak Q_j\|_{W^{1,p}(\mathbb H)}.
\end{align*}
By \eqref{eq:MS} we have
\begin{align}\label{eq:MS'}
\|\mathcal T\phi_j\|_{\mathcal M^{2,p}(\mathbb H)}\lesssim \|\varphi_j\|_{\mathcal M^{2-1/p,p}(\mathbb H)},
\end{align}
where the right-hand side is conveniently by assumption.
Hence we have
\begin{align*}
\int_{\mathcal I}\Big(\|\mathscr S_6\bff\|^r_{L^{p}_x}+\|\mathscr S_7\bff\|_{L^{p}_x}^r\Big)\dt\leq\,\delta(\mathrm{Lip}({\Omega}))\sum_{j=1}^\ell\int_{\mathcal I}\big(\|\bfV_j\|_{W^{2,p}_x}^r+\|\mathfrak Q_j\|_{W^{1,p}_x}^r\big)\dt.
\end{align*}
Using again the estimates for the problem on the half space from \eqref{est:stokeshalf} as well as \eqref{eq:2901} and \eqref{eq:2901'} we conclude
\begin{align*}
\int_{\mathcal I}\Big(\|\mathscr S_6\bff\|^r_{L^{p}_x}+\|\mathscr S_7\bff\|^r_{L^{p}_x}\Big)\dt\leq\,\delta'(\mathrm{Lip}({\Omega}))\int_{\mathcal I}\|\bff\|^r_{L^{p}_x}\dt.
\end{align*}
Both $\delta(\mathrm{Lip}({\Omega}))$ and $\delta'(\mathrm{Lip}({\Omega}))$ can be chosen conveniently small in dependence on $\mathrm{Lip}({\Omega})$.
Similarly, we have
\begin{align*}
\int_{\mathcal I}\|\mathscr S_8\bff\|_{L^{p}_x}^r\dt\leq\,\delta(\mathrm{Lip}({\Omega}))\sum_{j=1}^\ell\int_{\mathcal I}\|\partial_t\bfV_j\|_{L^{p}_x}^r\dt\leq \delta'(\mathrm{Lip}({\Omega}))\int_{\mathcal I}\|\bff\|^r_{L^{p}_x}\dt
\end{align*}
using once more \eqref{est:stokeshalf} in the last step.
In conclusion, choosing first $\ell$ large enough and
 then $T$ small enough we can infer that
\begin{align}\label{eq:july24}
\int_{\mathcal I}\|\mathscr S\bff\|^r_{L^{p}_x}\dt\leq\,\tfrac{1}{4}\int_{\mathcal I}\|\bff\|_{L^{p}_x}^r\dt.
\end{align}
Now we are going to show the same for $(\partial_t-\Delta)\mathscr R_1$. We have
\begin{align*}
\Div\mathscr R_0\bff=\nabla\xi_0\cdot\bfU_0+\sum_{j=1}^\ell\nabla\xi_j\cdot\bfV_j+\sum_{j=1}^\ell\xi_j(\mathbb I_{3\times 3}-\bfB_j)^\top:\nabla\bfV_j
\end{align*}
such that
\begin{align*}
\|\partial_t\mathscr R_1\bff\|_{L^{p}_x}&\lesssim \|\Bog_{{\Omega}}(\nabla\xi_0\cdot\partial_t\bfU_0)\|_{L^{p}_x}+\sum_{j=1}^\ell\|\Bog_{{\Omega}}(\nabla\xi_j\cdot\partial_t\bfV_j)\|_{L^{p}_x}\\
&+ \sum_{j=1}^\ell\|\Bog_{\Omega}\big(\xi_j(\mathbb I_{3\times 3}-\mathbf{B}_j)^\top:\nabla\partial_t\bfV_j\big)\|_{L^{p}_x}\\
&\lesssim \|\nabla\xi_0\cdot\partial_t\bfU_0\|_{W^{-1,p}_x}+\sum_{j=1}^\ell\|\nabla\xi_j\cdot\partial_t\bfV_j\|_{W^{-1,p}_x}\\
&+ \sum_{j=1}^\ell\|\xi_j(\mathbb I_{3\times 3}-\bfB_j)^\top:\nabla\partial_t\bfV_j\|_{W^{-1,p}_x}\\
&=: (R)_1+(R)_2+(R)_3
\end{align*}
using continuity of the Bogovskii-operator, cf. \eqref{eq:bog}.
Since $\bfU_0$ solves \eqref{eq:stokeswhole} we infer from \eqref{est:stokeswhole} that
\begin{align*}
\int_{\mathcal I}(R)_1^r\dt\lesssim \int_{\mathcal I}\|\partial_t\bfU_0\|_{W^{-1,p}_x}^r\dt\lesssim T^{r/2}\int_{\mathcal I}\|\bff\|_{L^{p}_x}^r\dt.
\end{align*}
Similarly, we obtain from \eqref{est:stokeshalf}
\begin{align*}
\int_{\mathcal I}(R)_2^r\dt\lesssim \sum_{j=1}^\ell\int_{\mathcal I}\|\partial_t\bfV_j\|_{W^{-1,p}_x}^r\dt\lesssim\sum_{j=1}^\ell\int_{\mathcal I}\|\partial_t\bfU_j\|_{W^{-1,p}_x}^r\dt\lesssim T^{r/2}\int_{\mathcal I}\|\bff\|_{L^{p}_x}^r\dt
\end{align*}
using also $\bfV_j=\bfU_j\circ\bfPhi_j\circ \mathscr V_j$ and \eqref{J} as well as \eqref{eq:MSa}--\eqref{lem:9.4.1}. Finally, arguing again as in the estimates for $\mathscr S_7$ above, and using $\Div\mathbf B_j=0$ (which holds as a consequence of the Piola-identity)
\begin{align*}
\int_{\mathcal I}(R)_3^r\dt&\lesssim \sum_{j=1}^\ell\int_{\mathcal I}\|\xi_j(\mathbb I_{3\times 3}-\mathbf B_j)^\top:\nabla\partial_t\bfV_j\|_{W^{-1,p}_x}^r\dt\\&\lesssim  \sum_{j=1}^\ell\int_{\mathcal I}\|\mathbb I_{3\times 3}-\mathbf{B}_j\|_{\mathcal M^{1,p}({\Omega})}^r\|\partial_t\bfV_j\|_{L^{p}_x}^r\dt\\
&\lesssim  \delta(\mathrm{Lip}(\partial{\Omega}))\int_{\mathcal I}\|\bff\|_{L^{p}_x}^r\dt.
\end{align*}
In conclusion, we have shown
\begin{align}\label{eq:july28}
\int_{\mathcal I}\|\partial_t\mathscr R_1\bff\|_{L^{p}_x}^r\dt&\leq\,\tfrac{1}{4}\int_{\mathcal I}\|\bff\|_{L^{p}_x}^r\dt,
\end{align}
for $T$ and $\mathrm{Lip}(\partial\Omega)$ sufficiently small. As far 
as $\Delta\mathscr R_1\bff$ is concerned, we have analogously
\begin{align*}
\int_{\mathcal I}\|\Delta\mathscr R_1\bff\|_{L^{p}_x}^r\dt&\lesssim \int_{\mathcal I}\|\Bog_{{\Omega}}(\nabla\xi_0\cdot\bfU_0)\|_{W^{2,p}_x}^r\dt+\sum_{j=1}^\ell\int_{\mathcal I}\|\Bog_{{\Omega}}(\nabla\xi_j\cdot\bfV_j)\|_{W^{2,p}_x}^r\dt\\
&+ \sum_{j=1}^\ell\int_{\mathcal I}\|\Bog_{\Omega}\big(\xi_j(\mathbb I_{3\times 3}-\mathbf{B}_j)^\top:\nabla\bfV_j\big)\|_{W^{2,p}_x}^r\dt\\
&\lesssim \int_{\mathcal I}\|\nabla\xi_0\cdot\bfU_0\|_{W^{1,p}_x}^r\dt+\sum_{j=1}^\ell\int_{\mathcal I}\|\nabla\eta_j\cdot\bfV_j\|_{W^{1,p}_x}^r\dt\\
&+ \sum_{j=1}^\ell\int_{\mathcal I}\|\xi_j(\mathbb I_{3\times 3}-\mathbf{B}_j)^\top:\nabla\bfV_j\|_{W^{1,p}_x}^r\dt\\
&\lesssim \int_{\mathcal I}\|\bfU_0\|_{W^{1,p}_x}^r\dt+\sum_{j=1}^\ell\int_{\mathcal I}\|\bfU_j\|_{W^{1,p}_x}^r\dt\\
&+ \sum_{j=1}^\ell\int_{\mathcal I}\|\mathbb I_{3\times 3}-\mathbf{B}_j\|^r_{\mathcal M^{1,p}({\Omega})}\|\bfU_j\|_{W^{2,p}_x}^r\dt\\
&\lesssim T^{r/2}\int_{\mathcal I}\|\bff\|_{L^{p}_x}^r\dt+ \delta(\mathrm{Lip}(\partial{\Omega}))\int_{\mathcal I}\|\bff\|_{L^{p}_x}^r\dt.
\end{align*}
Note that we also made use of \eqref{eq:2901}.
This implies
\begin{align}\label{eq:july28B}
\int_{\mathcal I}\|\Delta\mathscr R_1\bff\|_{L^{p}_x}^r\dt&\leq\,\tfrac{1}{4}\int_{\mathcal I}\|\bff\|_{L^{p}_x}^r\dt
\end{align}
choosing $T$ and $\mathrm{Lip}(\partial\Omega)$  small enough.
Combining \eqref{eq:july24}, \eqref{eq:july28} and \eqref{eq:july28B} implies $\|\mathscr T\|\leq\tfrac{3}{4}$. Hence $\mathscr L$ is onto recalling \eqref{eq:IS}. This means we have shown the claim for $T$ sufficiently small, say $T=T_0\ll1$. It is easy to extend it to the whole interval.
Let $(\bfu,\pi)$ be the solution in $[0,T]$. In order to obtain a solution on the whole interval we consider a partition of unity
$(\psi_k)_{k=1}^K$ on $[0,T]$ such that $\mathrm{spt}(\psi_k)\subset(\alpha_k,\alpha_k+T_0]$ for some $(\alpha_k)_{k=2}^K\subset[0,T]$ and $\alpha_1=0$. The functions $(\bfu_k,\pi_k)=(\psi_k\bfu,\psi_k\pi)$ are the unique solutions to
\begin{align*}
\partial_t\bfu_k=\Delta \bfu_k-\nabla\pi_k+\bff+\psi_k'\bfu,\quad\Div\bfu_k=0,\quad\bfu_k|_{(\alpha_k,\alpha_k+T_0)\times\partial{\Omega}}=0,\quad \bfu_k(\alpha_k,\cdot)=0.
\end{align*}
Applying the result proved for the interval $[0,T_0]$ and noticing that $\bfu=\sum_{k=1}^K\bfu_k$ and $\pi=\sum_{k=1}^K\pi_k$ proves the claim in the general case.
\end{proof}

\section{The perturbed system}\label{sec:pert}
In this section we develop a theory for some perturbed Stokes and Navier--Stokes systems which arise from the original one by flattening the boundary (introducing local coordinates as in Section \ref{sec:para}).
The perturbed Navier--Stokes system will be the basis for the partial regularity proof in Section \ref{sec:blowup} in which we compare its solution locally to a solution to the perturbed Stokes system.
By means of Sobolev multipliers we now derive optimal assumptions concerning the coefficients in the latter allowing for a maximal  regularity theory.

\subsection{Perturbed Navier--Stokes equations}
For a boundary suitable weak solution $(\bfu,\pi,\varphi)$ to \eqref{1} we define $\overline \pi=\pi\circ\bfPhi$, $\overline{\bfu}=\bfu\circ\bfPhi$ and  $\overline{\bff}=\bff\circ\bfPhi$, where $\bfPhi$ is the extension of $\varphi$ given in \eqref{eq:Phi}. We also introduce
\begin{align}\label{eq:AB}
\bfA_\varphi&=J_\varphi\big(\nabla \bfPsi\circ\bfPhi\big)^{\top}\nabla \bfPsi\circ\bfPhi,\quad\bfB_\varphi=J_\varphi\nabla \bfPsi\circ\bfPhi,
\end{align}
where $J_\varphi=\mathrm{det}(\nabla\bfPhi)$.
We see that $(\overline\bfu,\overline\pi,\varphi)$ is a solution to the system
\begin{align}\label{momref}
J_{\varphi}\partial_t\overline\bfu+(\bfB_\varphi\nabla\overline\bfu)\overline\bfu+\Div\big(\bfB_{\varphi}\overline\pi\big)-\Div\big(\bfA_{\varphi}\nabla\overline\bfu\big)&=J_\varphi\overline\bff,\\
\label{divref}\bfB_{\varphi}^\top:\nabla\overline\bfu=0,\quad\overline\bfu|_{\partial \mathcal B_1^+\cap\partial\mathbb H}&=0,
\end{align}
a.a. in $\mathcal Q_1^+$. 
Note that it may be necessary to translate and scale the coordinates in space-time to arrive at a system posed in $\mathcal Q_1^+:=\mathcal Q_1(0,0)$ (rather than in $\mathcal Q_r(t_0,x_0)$ for some $r>0$, $t_0\in \mathcal I$ and $x_0\in\partial\mathbb H$).
 Similarly, we can transform the local energy inequality leading to
\begin{equation}\label{energylocal}
\begin{split}
\int_ {\mt}\frac{1}{2}J_\varphi\zeta\big| \overline\bfu(t)\big|^2\dx&+\int_0^t\int_ {\mt}\zeta\bfA_\varphi\nabla \overline\bfu:\nabla \overline\bfu\dx\ds\\& \leq\int_0^t\int_{\mt}\frac{1}{2}J_\varphi |\overline\bfu|^2\partial_t\zeta\dx\ds+\int_0^t\int_{\mt}\frac{1}{2}J_\varphi| \overline\bfu|^2\Delta\bfPsi\circ\bfPhi\cdot\nabla\zeta\dx\ds\\&+\int_0^t\int_{\mt}\frac{1}{2}| \overline\bfu|^2\bfA_\varphi:\nabla^2\zeta\dx\ds+\int_0^t\int_{\mt}\frac{1}{2}\big(|\overline\bfu|^2+2\overline\pi\big)\overline\bfu\cdot\bfB_\varphi\nabla\zeta\dx\ds\\
&+\int_0^t\int_{\mt}J_\varphi\zeta\overline\bff\cdot\overline\bfu\dxs
\end{split}
\end{equation}
for any $\zeta\in C^\infty_c(\mathcal Q_1)$ with $\zeta\geq0$.

\begin{definition}[Boundary suitable weak solution perturbed system] \label{def:weakSolutionflat}
Let $(\overline\bff, \overline\bfu_0)$ be a dataset such that
\begin{equation}
\begin{aligned}
\label{dataset}
&\overline\bff \in L^{r_\ast}\big(\mathcal I_1; L^{s_\ast}(\mathcal B_1^+))\big),\quad 
\overline\bfu_0\in W^{2,s_\ast}\cap W^{1,2}_{0,\mathrm{\Div}}(\mathcal B_1^+).
\end{aligned}
\end{equation} 
We call the triple
$(\overline\bfu,\overline\pi,\varphi)$
a boundary suitable weak solution to the perturbed Navier--Stokes system \eqref{momref} with data $(\overline\bff, \overline\bfu_0)$ provided that the following holds:
\begin{itemize}
\item[(a)] The velocity field $\overline\bfu$ satisfies
\begin{align*}
 \overline\bfu \in L^\infty \big(\mathcal I_1; L^2(\mathcal B_1^+) \big)\cap  L^2 \big(\mathcal I_1; W^{1,2}_{\Div}(\mathcal B_1^+) \big) \cap L^{r_\ast}(\mathcal I_1;W^{2,s_\ast}(\mathcal B_1^+))\cap W^{1,r_\ast}(\mathcal I_1;L^{s_\ast}(\mathcal B_1^+)).
\end{align*}
\item[(b)] The pressure $\overline\pi$ satisfies
$$\overline\pi\in L^{r_\ast}(\mathcal I_1;W^{1,s_\ast}_\perp(\mathcal B_1^+)).$$
\item[c)] The boundary coordinates satisfy
\eqref{eq:varphi}.
\item[(d)] The system \eqref{momref}--\eqref{divref}
holds a.a. in $\mathcal Q^+_1$.
\item[(e)] The local energy inequality \eqref{energylocal} holds for any $\zeta\in C_c^\infty(\mathcal Q_1)$ with $\zeta\geq0$.
\end{itemize}
\end{definition}
A crucial part of the partial regularity proof in Section \ref{sec:blowup} will be the comparison of a boundary suitable weak solution of the perturbed Navier--Stokes system with a solution of the perturbed Stokes system. The analysis of the latter is the content of the following subsection.

\subsection{Perturbed Stokes equations}
In this section we consider, in analogy to \eqref{momref}--\eqref{divref}, a perturbed Stokes system in $\mathcal Q_1^+$ of the form
\begin{align}\label{eq:pertstokes}
\begin{aligned}
J_{\varphi}\partial_t\overline\bfu+\Div\big(\bfB_{\varphi}\overline{\pi}\big)-\Div\big(\bfA_{\varphi}\nabla\overline\bfu\big)&=\overline \bfg,\\
\bfB_{\varphi}^\top:\nabla\overline{\bfu}=\overline h,\quad\overline\bfu|_{\mathcal B^+_{1}\cap\partial\mathbb H}&=0,\quad\overline{\bfu}(-1,\cdot)=0,
\end{aligned}
\end{align}
where $\bfA_\varphi$ and $\bfB_\varphi$ are given in accordance with \eqref{eq:AB} for a given function $\varphi:\R^{2}\rightarrow\R$ and $\overline \bfg$ and $\overline h$ are given data.
We obtain the following maximal regularity result.
\begin{lemma}\label{lem:31}
Let $p,r\in(1,\infty)$.
 Suppose that $\varphi\in \mathcal M^{2-1/p,p}(\R^{2})(\delta)$ and that $\|\varphi\|_{W^{1,\infty}_y}\leq \delta$ 
 for some sufficiently small $\delta$. Assume further that $\overline \bfg\in L^r(\mathcal I_1;L^{p}(\mathcal B_1^+))$ and $\overline h\in L^r(\mathcal I_1;W^{1,p}\cap L^p_\perp(\mathcal B_1^+))$ with $\partial_t \overline h\in L^r(\mathcal I_1;W^{-1,p}(\mathcal B_1^+))$ and $\overline h(0,\cdot)=0$. Then there is a unique solution $(\overline\bfu,\overline\pi)$ to \eqref{eq:pertstokes} which satisfies
\begin{align}\label{eq:mainpara}
\begin{aligned}
\|\partial_t\overline\bfu\|_{L^r(\mathcal I_1;L^p(\mathcal B^+_{1}))}&+\|\overline\bfu\|_{L^r(\mathcal I_1;W^{2,p}(\mathcal B^+_{1}))}
+\|\overline\pi\|_{L^r(\mathcal I_1;W^{1,p}(\mathcal B^+_{1}))}\\&\lesssim \|\overline \bfg\|_{L^r(\mathcal I_1;L^p(\mathcal B^+_{1}))}+\|\nabla \overline h\|_{L^r(\mathcal I_1;L^{p}(\mathcal B^+_{1}))}+\|\partial_t \overline h\|_{L^r(\mathcal I_1;W^{-1,p}(\mathcal B^+_{1}))},
\end{aligned}
\end{align}
where the hidden constant only depends on $p,r$ and $\delta$.
\end{lemma}
\begin{proof}
Let us initially assume that $\overline\bfu$ and $\overline\pi$ are sufficiently smooth. We rewrite \eqref{eq:pertstokes} as
\begin{align}\label{eq:pertstokes'}
\begin{aligned}
\partial_t\overline\bfu+\nabla\overline{\pi}-\Delta\overline\bfu&=\overline \bfg+(1-J_\varphi)\partial_t\overline\bfu\\&+\Div\big((\mathbb I_{3\times 3}-\bfB_{\varphi})\overline{\pi}\big)+\Div\big((\bfA_{\varphi}-\mathbb I_{3\times 3})\nabla\overline\bfu\big),\\
\Div\overline{\bfu}=\big(\mathbb I_{3\times 3}-\bfB_{\varphi}\big)^\top&:\nabla\overline{\bfu}+\overline h,\quad\overline\bfu|_{\mathcal B^+_{1}\cap\partial\mathbb H}=0,\quad\overline{\bfu}(-1,\cdot)=0.
\end{aligned}
\end{align}
This is a classical Stokes system in $\mathcal Q_1^+$ with right-hand side
\begin{align*}
\overline \bfG:=\overline \bfg+(1-J_\varphi)\partial_t\overline\bfu+\Div\big((\mathbb I_{3\times 3}-\bfB_{\varphi})\overline{\pi}\big)+\Div\big((\bfA_{\varphi}-\mathbb I_{3\times 3})\nabla\overline\bfu\big)
\end{align*}
and given divergence
\begin{align*}
\overline H:=\big(\mathbb I_{3\times 3}-\bfB_{\varphi}\big)^\top&:\nabla\overline{\bfu}+\overline h.
\end{align*}
The known regularity theory (see \cite{FaSo} or \cite{So2}) yields
\begin{align}\label{eq:1101}\begin{aligned}
\|\partial_t\overline\bfu\|_{L^r(\mathcal I_1;L^p(\mathcal B^+_{1}))}&+\|\overline\bfu\|_{L^r(\mathcal I_1;W^{2,p}(\mathcal B^+_{1}))}
+\|\overline\pi\|_{L^r(\mathcal I_1;W^{1,p}(\mathcal B^+_{1}))}\\&\lesssim \|\overline \bfG\|_{L^r(\mathcal I_1;L^p(\mathcal B^+_{1}))}+\|\nabla \overline H\|_{L^r(\mathcal I_1;L^{p}(\mathcal B^+_{1}))}+\|\partial_t \overline H\|_{L^r(\mathcal I_1;W^{-1,p}(\mathcal B^+_{1}))}.
\end{aligned}
\end{align}
The goal is now to estimate the norms of $\overline \bfG$ and $\overline H$ employing the theory of Sobolev multipliers. First of all, we deduce from \eqref{est:ext} and \eqref{J} that
\begin{align*}
\|(1-J_\varphi)\partial_t\overline\bfu\|_{L^r(\mathcal I_1;L^p(\mathcal B^+_{1})}\lesssim \delta \|\partial_t\overline\bfu\|_{L^r(\mathcal I_1;L^p(\mathcal B^+_{1}))}
\end{align*}
using the assumption of a small Lipschitz constant.

Arguing similarly to the proof of Theorem \ref{thm:stokesunsteady} we have
\begin{align*}
\|\Div\big((\bfA_{\varphi}-\mathbb I_{3\times 3})\nabla\overline\bfu\big)\|_{L^p(\mathcal B_1^+)}&\lesssim \|\mathcal T\varphi\|_{\mathcal M(W^{2,p}(\mathcal B_1^+))}\|\overline{\bfu}\|_{W^{2,p}(\mathcal B_1^+)}
\end{align*}
%
as well as
\begin{align*}
\|\Div\big((\mathbb I_{3\times 3}-\bfB_{\varphi})\overline{\pi}\big)\|_{L^{p}(\mathcal B_1^+)}&\lesssim \|\bfB_\varphi-\mathbb I_{3\times 3}\|_{\mathcal M^{1,p}(\mathcal B_1^+)}\|\overline\pi\|_{W^{1,p}(\mathcal B_1^+)}\\
&\lesssim \|\mathcal T\varphi\|_{\mathcal M^{2,p}(\mathbb H)}\|\overline\pi\|_{W^{1,p}(\mathcal B_1^+)}.
\end{align*}
By \eqref{eq:MS} we have
\begin{align}\label{eq:MS'}
\|\mathcal T\varphi\|_{\mathcal M^{2,p}(\mathbb H)}\lesssim \|\varphi\|_{\mathcal M^{2-1/p,p}(\R^{n-1})}\leq \delta
\end{align}
using our assumption in the last step.
We conclude that
\begin{align*}
\|\bfG\|_{L^r(\mathcal I_1;L^p(\mathcal B^+_{1}))}&\lesssim \|\overline\bfg\|_{L^r(\mathcal I_1;L^p(\mathcal B^+_{1}))}+\delta\big(\|\partial_t\overline\bfu\|_{L^r(\mathcal I_1;L^p(\mathcal B^+_{1}))}+\|\overline\bfu\|_{L^r(\mathcal I_1;W^{2,p}(\mathcal B^+_{1}))}\big)\\
&+\delta\|\overline\pi\|_{L^r(\mathcal I_1;W^{1,p}(\mathcal B^+_{1}))}.
\end{align*}
By analogous arguments we obtain
\begin{align*}
\|\nabla\big((\mathbb I_{3\times 3}-\bfB_{\varphi})^\top:\nabla\overline{\bfu}\big)\|_{L^{p}(\mathcal B_1^+)}&\lesssim \|\bfB_\varphi-\mathbb I_{3\times 3}\|_{\mathcal M^{1,p}(\mathcal B_1^+)}\|\nabla\overline\bfu\|_{W^{1,p}(\mathcal B_1^+)}\\
&\lesssim \|\mathcal T\varphi\|_{\mathcal M^{2,p}(\mathbb H)}\|\overline\bfu\|_{W^{2,p}(\mathcal B_1^+)},\\
\|\partial_t\big((\mathbb I_{3\times 3}-\bfB_{\varphi})^\top:\nabla\overline{\bfu}\big)\|_{W^{-1,p}(\mathcal B_1^+)}&=\|(\mathbb I_{3\times 3}-\bfB_{\varphi})^\top:\nabla\partial_t\overline{\bfu}\|_{W^{-1,p}(\mathcal B_1^+)}\\
&\leq  \|\bfB_\varphi-\mathbb I_{3\times 3}\|_{\mathcal M^{1,p}(\mathcal B_1^+)}\|\partial_t\nabla\overline\bfu\|_{W^{-1,p}(\mathcal B_1^+)}\\
&\lesssim \|\mathcal T\varphi\|_{\mathcal M^{2,p}(\mathbb H)}\|\partial_t\overline\bfu\|_{L^{p}(\mathcal B_1^+)},
\end{align*}
such that
\begin{align*}
\|\nabla \overline H\|_{L^r(\mathcal I_1;L^{p}(\mathcal B^+_{1}))}+\|\partial_t \overline H\|_{L^r(\mathcal I_1;W^{-1,p}(\mathcal B^+_{1}))}&\lesssim \|\nabla \overline h\|_{L^r(\mathcal I_1;L^{p}(\mathcal B^+_{1}))}+\|\partial_t\overline h\|_{L^r(\mathcal I_1;W^{-1,p}(\mathcal B^+_{1}))}\\
&+\delta\big(\|\partial_t\overline\bfu\|_{L^r(\mathcal I_1;L^p(\mathcal B^+_{1}))}+\|\overline\bfu\|_{L^r(\mathcal I_1;W^{2,p}(\mathcal B^+_{1}))}\big).
\end{align*}
Combing the estimates for $\overline \bfG$ and $\overline H$
yields the claim for $\delta$ sufficiently small as a consequence of \eqref{eq:1101}.

Since smoothness of $(\overline\bfu,\overline\pi)$ is not a priori known one has to regularise the equations. This can be done
by mollifying the function $\varphi$ with a standard mollifier.
As shown in \cite[Lemma 4.3.3.]{MaSh} mollification does not expand the $\mathcal M^{s,p}(\R^{2})$-norm.
 For the regularised problem the results from \cite[Lemma 3.1]{SeShSo} apply and we obtain a sufficiently smooth solution. The previous estimates can then be performed uniformly with respect to the mollification parameter and the claimed result follows in the limit.
\end{proof}
As in \cite[Lemma 3.2]{SeShSo} we deduce the following Cacciopoli-type inequality from Lemma \ref{lem:31}.
\begin{lemma}\label{lem:32}
Let $p,q,r\in(1,\infty)$ with $q\geq p$.
 Suppose that $\varphi\in \mathcal M^{2-1/q,q}(\R^{2})(\delta)$ and that $\|\varphi\|_{W^{1,\infty}_y}\leq \delta$ for some sufficiently small $\delta$. Assume further that $\overline\bfg\in L^r(\mathcal I_1;L^{p}(\mathcal B_1^+))$ and that $\overline h=0$.
The solution $(\overline\bfu,\overline\pi)$ to \eqref{eq:pertstokes} satisfies
\begin{align}\label{eq:mainpara'}
\begin{aligned}
\|\partial_t\overline\bfu\|_{L^r(\mathcal I_{1/2};L^q(\mathcal B^+_{1/2}))}&+\|\overline\bfu\|_{L^r(\mathcal I_{1/2};W^{2,q}(\mathcal B^+_{1/2}))}
+\|\overline\pi\|_{L^r(\mathcal I_{1/2};W^{1,q}(\mathcal B^+_{1/2}))}\\&\lesssim \|\overline\bfg\|_{L^r(\mathcal I_1;L^q(\mathcal B^+_{1}))}+\|\nabla \overline\bfu\|_{L^r(\mathcal I_1;L^{p}(\mathcal B^+_{1}))}+\|\overline\pi-(\overline\pi)_{\mathcal B_1^+}\|_{L^r(\mathcal I_1;L^{p}(\mathcal B^+_{1}))},
\end{aligned}
\end{align}
where the hidden constant only depends on $p,q,r$ and $\delta$.
\end{lemma}
\begin{proof}
Let us initially suppose that $p=q$.
We consider a cut-off function $\zeta\in C^\infty_c(\mathcal Q_1)$ with $\zeta=1$ in $\mathcal Q_{3/4}$. The functions $\overline\bfv=\zeta\overline\bfu$ and $\overline{\mathfrak q}=\zeta\overline\pi$ solve the system
\begin{align}\label{eq:pertstokes'}
\begin{aligned}
J_{\varphi}\partial_t\overline\bfv+\Div\big(\bfB_{\varphi}\overline{\pi}\big)-\Div\big(\bfA_{\varphi}\nabla\overline\bfv\big)&=\overline \bfG_\zeta,\\
\bfB_{\varphi}:\nabla\overline{\bfv}=\overline h_\zeta,\quad\overline\bfv|_{\mathcal B^+_{1}\cap\partial\mathbb H}&=0,\quad\overline{\bfv}(-1,\cdot)=0,
\end{aligned}
\end{align}
where $\overline\bfG_\zeta$ and $\overline h_\zeta$ are given by
\begin{align*}
\overline\bfG_\zeta&=\zeta\overline\bfg-J_\varphi\partial_t\zeta\overline\bfu-2\bfA_\varphi\nabla\overline\bfu\nabla\zeta
+(\overline\pi-(\overline\pi)_{\mathcal B_1^+})\bfB_\varphi\nabla\zeta-\overline\bfu\Div\big(\bfA_{\varphi}\nabla\zeta),\quad
\overline h_\zeta=\overline\bfu\cdot\bfB_\varphi\nabla\zeta.
\end{align*}
We apply Lemma \ref{lem:31} to \eqref{eq:pertstokes'} and obtain
\begin{align*}
\|\partial_t\overline\bfv\|_{L^r(\mathcal I_1;L^p(\mathcal B^+_{1}))}&+\|\overline\bfv\|_{L^r(\mathcal I_1;W^{2,p}(\mathcal B^+_{1}))}
+\|\overline{\mathfrak q}\|_{L^r(\mathcal I_1;W^{1,p}(\mathcal B^+_{1}))}\\&\lesssim \|\overline\bfG_\zeta\|_{L^r(\mathcal I_1;L^p(\mathcal B^+_{1}))}+\|\nabla \overline h_\zeta\|_{L^r(\mathcal I_1;L^{p}(\mathcal B^+_{1}))}+\|\partial_t \overline h_\zeta\|_{L^r(\mathcal I_1;W^{-1,p}(\mathcal B^+_{1}))}.
\end{align*}
We clearly have
\begin{align*}
\|\overline\bfG_\zeta\|_{L^r(\mathcal I_1;L^p(\mathcal B^+_{1}))}&\lesssim \|\overline\bfg\|_{L^r(\mathcal I_1;L^p(\mathcal B^+_{1}))}+\|\nabla \overline\bfu\|_{L^r(\mathcal I_1;L^{p}(\mathcal B^+_{1}))}+\|\overline\pi-(\overline\pi)_{\mathcal B_1^+}\|_{L^r(\mathcal I_1;L^{p}(\mathcal B^+_{1}))}\\&+ \|\overline\bfu\Div\big(\bfA_{\varphi}\nabla\zeta)\|_{L^r(\mathcal I_1;L^{p}(\mathcal B^+_{1}))}
\end{align*}
since $\varphi$ is Lipschitz by assumption (and so are $\bfPhi$ and $\bfPsi$, cf. \eqref{est:ext} and \eqref{eq:detJ}). Note that we also used Poincar\'e's inequality recalling that $\overline\bfu|_{\mathcal B^+_{1}\cap\partial\mathbb H}=0$. For the last term in the above we have
\begin{align*}
\|\overline\bfu\Div\big(\bfA_{\varphi}\nabla\zeta)\|_{L^r(\mathcal I_1;L^{p}(\mathcal B^+_{1}))}&\lesssim \|\Div\big(\overline\bfu\otimes\bfA_{\varphi}\nabla\zeta)\|_{L^r(\mathcal I_1;L^{p}(\mathcal B^+_{1}))}+\|\nabla\overline\bfu\,\bfA_{\varphi}\nabla\zeta\|_{L^r(\mathcal I_1;L^{p}(\mathcal B^+_{1}))}\\
&\lesssim \|\overline\bfu\otimes\bfA_{\varphi}\nabla\zeta\|_{L^r(\mathcal I_1;W^{1,p}(\mathcal B^+_{1}))}+\|\nabla\overline\bfu\|_{L^r(\mathcal I_1;L^{p}(\mathcal B^+_{1}))}\\
&\lesssim \|\bfA_\varphi\|_{\mathcal M^{1,p}(\mathcal B^+_1)}\|\overline\bfu\otimes\nabla\zeta\|_{L^r(\mathcal I_1;W^{1,p}(\mathcal B^+_{1}))}+\|\nabla\overline\bfu\|_{L^r(\mathcal I_1;L^{p}(\mathcal B^+_{1}))}\\
&\lesssim \|\nabla\overline\bfu\|_{L^r(\mathcal I_1;L^{p}(\mathcal B^+_{1}))}.
\end{align*}
Note that we used in the last step, in addition to the Lipschitz continuity of $\bfPhi$ and $\bfPsi$, that
$\bfPhi$ and $\bfPsi$ belong to the correct mulitplier space by \eqref{eq:SoMo} and \eqref{eq:SMPhiPsi} using
the assumption $\varphi\in \mathcal M^{2-1/p,p}(\R^{2})(\delta)$. Similarly, we have
\begin{align*}
\|\overline h\|_{L^r(\mathcal I_1;W^{1,p}(\mathcal B^+_{1}))}
&\lesssim \|\overline\bfu\cdot\bfB_{\varphi}\nabla\zeta\|_{L^r(\mathcal I_1;W^{1,p}(\mathcal B^+_{1}))}\\
&\lesssim \|\bfB_\zeta\|_{\mathcal M^{1,p}(\mathcal B^+_1)}\|\overline\bfu\otimes\nabla\zeta\|_{L^r(\mathcal I_1;W^{1,p}(\mathcal B^+_{1}))}\\
&\lesssim \|\nabla\overline\bfu\|_{L^r(\mathcal I_1;L^{p}(\mathcal B^+_{1}))}.
\end{align*}
Finally,
\begin{align*}
\|\partial_t \overline h\|_{L^r(\mathcal I_1;W^{-1,p}(\mathcal B^+_{1}))}&\lesssim \|\partial_t \overline\bfu\|_{L^r(\mathcal I_1;W^{-1,p}(\mathcal B^+_{1}))}+\|\overline\bfu\|_{L^r(\mathcal I_1;W^{-1,p}(\mathcal B^+_{1}))}.
\end{align*}
Using equation \eqref{eq:pertstokes'} together with strict positivity of $J_\varphi$ and $J_\varphi\in \mathcal M^{1,p}(\mathcal B^+_1)$ we have
\begin{align*}
 \|\partial_t \overline\bfu\|_{L^r(\mathcal I_1;W^{-1,p}(\mathcal B^+_{1}))}&\lesssim \|\overline\bfG\|_{L^r(\mathcal I_1;W^{-1,p}(\mathcal B^+_{1}))}+\|\Div\big(\bfB_{\varphi}\overline{\pi}\big)\|_{L^r(\mathcal I_1;W^{-1,p}(\mathcal B^+_{1}))}\\&+\|\Div\big(\bfA_{\varphi}\nabla\overline\bfu\big)\|_{L^r(\mathcal I_1;W^{-1,p}(\mathcal B^+_{1}))}\\
 &\lesssim \|\overline\bfG\|_{L^r(\mathcal I_1;L^p(\mathcal B^+_{1}))}+\|\nabla \overline\bfu\|_{L^r(\mathcal I_1;L^{p}(\mathcal B^+_{1}))}+\|\overline\pi-(\overline\pi)_{\mathcal B_1^+}\|_{L^r(\mathcal I_1;L^{p}(\mathcal B^+_{1}))}.
\end{align*}
Note that we used again boundedness of $\bfA_\varphi$ and $\bfB_\varphi$.
Combining everything, we have proved \eqref{eq:mainpara'} for the case $p=q$ (even with norms over $\mathcal Q^+_{3/4}$ on the left-hand side).

Let us now consider the case $q\in(p,\tfrac{3p}{3-p}]$ for the case $p<3$ and $q\in(p,\infty)$ arbitrary for $p\geq 3$. We make the same definitions of
$\overline\bfv$ and $\overline{\mathfrak q}$ as above but consider a cut-off function $\zeta\in C^\infty_c(\mathcal{Q}_{3/4})$ with $\zeta=1$ in $\mathcal Q_{1/2}$. By the choice of $q$ and Sobolev's embedding $W^{1,p}(\mathcal B_{3/4}^+)\hookrightarrow L^{q}(\mathcal B_{3/4}^+)$ we have
\begin{align*}
\|\overline\bfG_\zeta\|_{L^r(\mathcal I_{3/4};L^q(\mathcal B^+_{3/4}))}&\lesssim \|\overline\bfg\|_{L^r(\mathcal I_{3/4};L^q(\mathcal B^+_{3/4}))}+\|\nabla \overline\bfu\|_{L^r(\mathcal I_{3/4};L^{q}(\mathcal B^+_{3/4}))}\\&+\|\overline\pi-(\overline\pi)_{\mathcal B_{3/4}^+}\|_{L^r(\mathcal I_{3/4};L^{q}(\mathcal B^+_{3/4}))}\\
&\lesssim \|\overline\bfg\|_{L^r(\mathcal I_{3/4};L^q(\mathcal B^+_{3/4}))}+\|\nabla^2 \overline\bfu\|_{L^r(\mathcal I_{3/4};L^{p}(\mathcal B^+_{3/4}))}+\|\nabla\overline\pi\|_{L^r(\mathcal I_{3/4};L^{p}(\mathcal B^+_{3/4}))},
\end{align*}
arguing as for the $L^p$-estimates above (but using $\varphi\in \mathcal M^{2-1/q,q}(\R^{2})(\delta)$).
Similarly, we obtain
\begin{align*}
\|\nabla \overline h_\zeta\|_{L^r(\mathcal I_{3/4};L^{q}(\mathcal B^+_{3/4}))}&\lesssim\|\nabla^2\overline\bfu\|_{L^r(\mathcal I_{3/4};L^{p}(\mathcal B^+_{3/4}))},\\
\|\partial_t \overline h_\zeta\|_{L^r(\mathcal I_{3/4};W^{-1,q}(\mathcal B^+_{3/4}))}&\lesssim\|\partial_t\overline\bfu\|_{L^r(\mathcal I_{3/4};L^{p}(\mathcal B^+_{3/4}))}+\|\overline\bfu\|_{L^r(\mathcal I_{3/4};L^{p}(\mathcal B^+_{3/4}))}.
\end{align*}
Applying now the result from the first step as well as Lemma \ref{lem:31} (but with the stronger assumption $\varphi\in \mathcal M^{2-1/q,q}(\R^{2})(\delta)$,  see \cite[Proposition 3.5.3.]{MaSh} for inclusions of the spaces $\mathcal M^{s,p}$) yields again the claim. Note that we still have the restriction $q\leq \tfrac{3p}{3-p}$ if $p<3$.
Iterating this argument based on Sobolev's embedding and consider appropriate cut-off functions clearly allows any choice of exponent $q<\infty$.
\end{proof}

We will also need interior estimates for points close to the boundary which is why we consider 
in analogy to \eqref{eq:pertstokes} the system
 \begin{align}\label{eq:pertstokesint}
\begin{aligned}
J_{\varphi}\partial_t\overline\bfu+\Div\big(\bfB_{\varphi}\overline{\pi}\big)-\Div\big(\bfA_{\varphi}\nabla\overline\bfu\big)&=\bfg,\\
\bfB_{\varphi}^\top:\nabla\overline{\bfu}=h,\quad\overline\bfu|_{\mathcal B_{1}}&=0,\quad\overline{\bfu}(-1,\cdot)=0,
\end{aligned}
\end{align}
in $\mathcal Q_1$. Arguing exactly as for Lemmas \ref{lem:31} and \ref{lem:32} we obtain the following results.
\begin{lemma}\label{lem:31int}
Let $p,r\in(1,\infty)$.
 Suppose that $\varphi\in \mathcal M^{2-1/p,p}(\R^{2})(\delta)$ and that $\|\varphi\|_{W^{1,\infty}_y}\leq \delta$ for some sufficiently small $\delta$. Assume further that $\overline \bfg\in L^r(\mathcal I_1;L^{p}(\mathcal B_1))$ and $\overline h\in L^r(\mathcal I_1;W^{1,p}(\mathcal B_1))$ with $\partial_t \overline h\in L^r(\mathcal I_1;W^{-1,p}(\mathcal B_1))$. Then there is a unique solution $(\overline\bfu,\overline\pi)$ to \eqref{eq:pertstokesint} which satisfies
\begin{align}\label{eq:mainparaint}
\begin{aligned}
\|\partial_t\overline\bfu\|_{L^r(\mathcal I_1;L^p(\mathcal B_{1}))}&+\|\overline\bfu\|_{L^r(\mathcal I_1;W^{2,p}(\mathcal B_{1}))}
+\|\overline\pi\|_{L^r(\mathcal I_1;W^{1,p}(\mathcal B_{1}))}\\&\lesssim \|\overline\bfg\|_{L^r(\mathcal I_1;L^p(\mathcal B_{1}))}+\|\nabla \overline h\|_{L^r(\mathcal I_1;L^{p}(\mathcal B_{1}))}+\|\partial_t \overline h\|_{L^r(\mathcal I_1;W^{-1,p}(\mathcal B_{1}))},
\end{aligned}
\end{align}
where the hidden constant only depends on $p,r$ and $\delta$.
\end{lemma}
\begin{lemma}\label{lem:32int}
Let $p,q,r\in(1,\infty)$ with $q\geq p$.
 Suppose that $\varphi\in \mathcal M^{2-1/q,q}(\R^{2})(\delta)$ and that $\|\varphi\|_{W^{1,\infty}_y}\leq \delta$ for some sufficiently small $\delta$. Assume further that $\overline\bfg\in L^r(\mathcal I_1;L^{p}(\mathcal B_1))$ and that $\overline h=0$.
The solution $(\overline\bfu,\overline\pi)$ to \eqref{eq:pertstokesint} satisfies
\begin{align}\label{eq:mainpara'int}
\begin{aligned}
\|\partial_t\overline\bfu\|_{L^r(\mathcal I_{1/2};L^q(\mathcal B_{1/2}))}&+\|\overline\bfu\|_{L^r(\mathcal I_{1/2};W^{2,q}(\mathcal B_{1/2}))}
+\|\overline\pi\|_{L^r(\mathcal I_{1/2};W^{1,q}(\mathcal B_{1/2}))}\\&\lesssim \|\overline \bfg\|_{L^r(\mathcal I_1;L^q(\mathcal B_{1}))}+\|\nabla \overline\bfu\|_{L^r(\mathcal I_1;L^{p}(\mathcal B_{1}))}+\|\overline\pi-(\overline\pi)_{\mathcal B_1}\|_{L^r(\mathcal I_1;L^{p}(\mathcal B_{1}))},
\end{aligned}
\end{align}
where the hidden constant only depends on $p,q,r$ and $\delta$.
\end{lemma}

\section{Blow-up \& proof of Theorem \ref{thm:main}}
\label{sec:blowup}
We consider a boundary suitable weak solution $(\overline\bfu,\overline\pi,\varphi)$ to the perturbed Navier--Stokes system \eqref{momref}--\eqref{divref} with forcing $\overline\bff$ as defined in Definition
\ref{def:weakSolutionflat}.
We define
 the excess-functional as
\begin{align*}
\mathscr E_r^{t_0,x_0} (\overline\bfu,\overline\pi)  &:=  \dashint_{\mathcal Q^+_{r}(t_0,x_0)} |\overline\bfu|^3 \dy\ds+ r^3\bigg(\dashint_{\mathcal Q^+_{r}(t_0,x_0)}  |\overline\pi - (\overline\pi)_{\mathcal B^+_r(x_0)}|^{5/3} \dy\ds\bigg)^{\frac{9}{5}}
\end{align*}
and use the short-hand notation $\mathscr E_r(\overline\bfu,\overline\pi)=:\mathscr E_r^{0,0} (\overline\bfu,\overline\pi)$.
The following lemma is the bulk of the partial regularity proof and the claim of Theorem \ref{thm:main} will follow in a standard manner.
\begin{lemma}
\label{Lemma}
Let $p>\frac{15}{4}$
and $\overline\bff\in L^{p}(\mathcal I_1;L^{p}(\mathcal B_1^+))$ be given. Suppose that $\delta>0$ is sufficienlty small.
For any $\tau\in(0,1/2)$ there exist constants
$\varepsilon>0$ (small) and $C_\ast>0$ (large) such the following implication is true for any
triple $(\overline\bfu,\overline\pi,\varphi)$ which is a boundary suitable weak solution to \eqref{momref}--\eqref{divref} in $\mathcal Q_1^+$ in the sense of Definition \ref{def:weakSolution}, where $\varphi\in\mathcal M^{2-1/p,p}(\R^2)(\delta)$ with $\|\varphi\|_{W^{1,\infty}_y}\leq\delta$: Suppose that
\begin{equation}
\label{I1}
\mathscr E_1(\overline\bfu,\overline\pi) + 
\bigg(\ddashint_{\mathcal Q^+_1}  |\overline\bff|^{p} \dy\ds\bigg)^{\frac{3}{p}}
\leq\varepsilon,
\end{equation}
then we have
\begin{equation}
\label{I2}
\mathscr E_\tau (\overline\bfu,\overline\pi) \le C_\ast \tau^{2\alpha} \bigg(\mathscr E_1 (\overline\bfu,\overline\pi) +\bigg(\ddashint_{\mathcal Q^+_1}  |\overline\bff|^{p} \dy\ds\bigg)^{\frac{3}{p}}\bigg),
\end{equation}
where $\alpha=3\big(\frac{2}{5}-\frac{3}{2p}\big)$.
\end{lemma}

\begin{proof} We argue by contradiction. Suppose there is $\tau \in (0,\tfrac{1}{2})$, and a sequence of boundary suitable weak solution $(\overline\bfu_m,\overline\pi_m,\varphi_m)$ such that 
\begin{align}\label{eq:varphim}
\|\nabla\varphi_m\|_{L^{\infty}_y}+\|\varphi_m\|_{\mathcal M^{2-1/p,p}(\R^2)}\leq \delta,\quad \varphi_m(0)=0,\quad \nabla\varphi_m(0)=0,
\end{align}
as well as
\begin{align}
\label{I3}
\lambda_m^3:=& \mathscr E_1(\overline\bfu_m,\overline\pi_m)+\|\overline\bff_m\|^3_{ L^{p}(\mathcal I_1;L^{p}(\mathcal B_1^+))} \to 0, \quad \ m \to \infty, \\
\label{I4}
&\mathscr E_\tau(\overline\bfu_m,\overline\pi_m) >  \tfrac{1}{2} \lambda^3_m .
\end{align}
We obtain from \eqref{eq:varphim}
\begin{align}\label{eq:varphim2}
\varphi_m\rightharpoonup^\ast\varphi\quad \text{in}\quad W^{1,\infty}(\R^{2}),\quad \varphi\in\mathcal M^{2-1/p,p}(\R^{2})(\delta),\quad \|\varphi\|_{W^{1,\infty}_y}\leq \delta.
\end{align}
Hence \eqref{est:ext} and \eqref{eq:MS} yields
\begin{align}\label{eq:Phim2}
\bfPhi_m\rightharpoonup^\ast\bfPhi\quad \text{in}\quad W^{1,\infty}(\R^3),\quad \bfPhi\in\mathcal M^{2,p}(\R^3),
\end{align}
where $\bfPhi_m$ is defined in accordance with \eqref{eq:Phi}. Finally, \eqref{eq:detJ} and \eqref{eq:SMPhiPsi} imply
\begin{align}\label{eq:Psim2}
\bfPsi_m\rightharpoonup^\ast\bfPsi\quad \text{in}\quad W^{1,\infty}(\R^3),\quad \bfPsi\in \mathcal M^{2,p}(\R^3),
\end{align}
where $\bfPsi_m$ is the inverse of $\bfPhi_m$.
We define in $\mathcal I_1\times\mathcal B^+_1$
\begin{align*}
\overline\bfv_m&:= \frac{1}{\lambda_m} \overline\bfu_m,\quad
\overline{\mathfrak q}_m:= \frac{1}{\lambda_m } \big[\overline\pi_m- (\overline\pi_m)_{\mathcal B_1^+}\big],\quad
\overline\bfg_m:=\frac{1}{\lambda_m}\overline\bff.
\end{align*}
We get from \eqref{I3} 
the relation
\begin{equation}
\label{I6} \dashint_{\mathcal Q^+_1} |\overline\bfv_{m}|^3 \dz\ds+ \bigg(\dashint_{\mathcal Q^+_1} |\overline{\mathfrak q}_{m}|^{5/3} \dz\ds\bigg)^{\frac{9}{5}}+ \bigg(\dashint_{\mathcal Q^+_1} |\overline\bfg_{m}|^{p} \dz\ds\bigg)^{\frac{3}{p}} =  1,
\end{equation}
such that, after passing to a subsequence,
\begin{align}
\label{I8}
 \overline{\mathfrak{q}}_m &\rightharpoonup  \overline{\mathfrak q} \quad \text{in} \quad L^{5/3} (\mathcal I_1;L^{5/3}(\mathcal B^+_1)),\\
\overline\bfv_m &\rightharpoonup  \overline\bfv\quad\text{in}\quad  L^3(\mathcal I_1;L^3(\mathcal B^+_1))\label{I8'},\\
\overline\bfg_m &\rightharpoonup  \overline\bfg\quad\text{in}\quad  L^p(\mathcal I_1;L^p(\mathcal B^+_1))\label{I8''}.
\end{align}
On the other hand, (\ref{I4}) reads after scaling
\begin{align}
\label{I7}
\dashint_{\mathcal Q_\tau^+} |\overline\bfv_{m} |^3 \dz\ds
 &+\tau^3\bigg(\dashint_{\mathcal I_1}\dashint_{\mathcal B_\tau^+}  |\overline{\mathfrak{q}}_m  - (\overline{\mathfrak q}_{m})_{\mathcal B_\tau^+} |^{5/3} \dz\ds\bigg)^{\frac{9}{5}}> C_{\ast} \tau^{2\alpha} .
\end{align}
In order to proceed we use a scaled version of the equation which reads as
\begin{align}\label{eq:eq}
\begin{aligned}
J_{\varphi_m}\partial_t\overline\bfv_m+\lambda_m(\bfB_{\varphi_m}\nabla\overline\bfv_m)\overline\bfv_m+\Div\big(\bfB_{\varphi_m}\overline{\mathfrak q}_m\big)-\Div\big(\bfA_{\varphi_m}\nabla\overline\bfv_m\big)&=J_{\varphi_m}\overline\bfg_m,\\
\bfB_{\varphi_m}:\nabla{\overline\bfv}_m=0,\quad\overline\bfv_{m}|_{\mathcal B^+_{1}\cap\partial\mathbb H}&=0.
\end{aligned}
\end{align}
By \eqref{energylocal} (replacing $\zeta$ by $\zeta^2$) we have
\begin{align*}
\int_ {\mathcal B^+_1}\frac{1}{2}J_{\varphi_m}\zeta^2&\big| \overline\bfv_m(t)\big|^2\dx+\int_0^t\int_ {\mathcal B^+_1}\zeta^2\bfA_{\varphi_m}\nabla \overline\bfv_m:\nabla \overline\bfv_m\dx\ds\\& \leq\int_0^t\int_{\mathcal B^+_1}\frac{1}{2}J_{\varphi_m} |\overline\bfv_m|^2\partial_t\zeta^2\dx\ds+\int_0^t\int_{\mathcal B^+_1}\frac{1}{2}J_{\varphi_m}| \overline\bfv_m|^2\Delta{\bfPsi_m}\circ\bfPhi_m\cdot\nabla\zeta^2\dx\ds\\&+\int_0^t\int_{\mathcal B^+_1}\frac{1}{2}| \overline\bfv_m|^2\bfA_{\varphi_m}:\nabla^2\zeta^2\dx\ds+\int_0^t\int_{\mathcal B^+_1}\frac{1}{2}\big(\lambda_m|\overline\bfv_m|^2+2\overline{\mathfrak q}_m\big)\overline\bfv_m\cdot\bfB_{\varphi_m}\nabla\zeta^2\Big)\dx\ds\\
&+\int_0^t\int_{\mt}J_{\varphi_m}\zeta^2\overline\bfg_m\cdot\overline\bfv_m\dxs
\end{align*}
for all non-negative $\zeta\in C^\infty_c(\mathcal Q_1)$.  Further, note that by \eqref{eq:Phim2} and \eqref{eq:Psim2} we know that
$J_{\varphi_m}$, $\bfA_{\varphi_m}$ and $\bfB_{\varphi_m}$ are uniformly bounded. Hence the first, third, fourth and fifth term are uniformly bounded by \eqref{I6}. The second term is more delicate and we need to employ Sobolev mutltipliers (see Section \ref{sec:SM} for a brief introduction). We obtain by \eqref{eq:Phim2}, \eqref{eq:Psim2} and \eqref{I6}
\begin{align*}
\int_0^t\int_{\mathcal B^+_1}&\frac{1}{2}J_{\varphi_m}| \overline{\bfv}_m|^2\Delta{\bfPsi_m}\circ\bfPhi_m\cdot\nabla\zeta^2\dx\ds\\&\lesssim \sum_{k=1}^3 \int_0^t\int_{\mathcal B^+_1}| \overline{\bfv}_m||\zeta\overline{\bfv}_m\Delta\bfPsi_m^k\circ\bfPhi_m|\dx\ds\\
&\lesssim \int_0^t\int_{\mathcal B^+_1}| \overline{\bfv}_m|^3\dx\ds+\sum_{k=1}^3 \int_0^t\int_{\mathcal B^+_1}|\zeta\overline{\bfv}_m\Delta\bfPsi_m^k\circ\bfPhi_m|^{3/2}\dx\ds\\
&\lesssim 1+\sum_{k=1}^3 \int_0^t\int_{\bfPhi_m(\mathcal B^+_1)}|(\zeta\overline{\bfv}_m)\circ\bfPsi_m\Delta\bfPsi_m^k|^{3/2}\dx\ds=:1+(I)_m.
\end{align*}
The remaining integral $(I)_m$ can be split into
\begin{align*}
(I)_m&\lesssim \sum_{k=1}^3 \int_0^t\int_{\bfPhi_m(\mathcal B^+_1)}|\nabla((\zeta\overline{\bfv}_m)\circ\bfPsi_m)\nabla\bfPsi_m^k|^{3/2}\dxs\\&+\sum_{k=1}^3 \int_0^t\int_{\bfPhi_m(\mathcal B^+_1)}|\Div((\zeta\overline{\bfv}_m)\circ\bfPsi_m\otimes\nabla\bfPsi_m^k)|^{3/2}\dxs=:(I)_m^1+(I)_m^2.
\end{align*}
By \eqref{eq:Phim2} and \eqref{eq:Psim2} we have for $\kappa>0$ arbitrary
\begin{align}\nonumber
(I)_m^1&\lesssim \sum_{k=1}^3 \int_0^t\int_{\bfPhi_m(\mathcal B^+_1)}|\nabla(\zeta\circ\bfPsi_m)\overline{\bfv}_m\circ\bfPsi_m\nabla\bfPsi_m^k|^{3/2}\dxs\\\nonumber
&+\sum_{k=1}^3 \int_0^t\int_{\bfPhi_m(\mathcal B^+_1)}|\zeta\circ\bfPsi_m\nabla(\overline{\bfv}_m\circ\bfPsi_m)\nabla\bfPsi_m^k|^{3/2}\dxs\\\nonumber
&\lesssim  \int_0^t\int_{\bfPhi_m(\mathcal B^+_1)}|\nabla(\zeta\circ\bfPsi_m)\overline{\bfv}_m\circ\bfPsi_m|^{3/2}\dxs\\\nonumber
&+ \int_0^t\int_{\bfPhi_m(\mathcal B^+_1)}|\zeta\circ\bfPsi_m\nabla(\overline{\bfv}_m\circ\bfPsi_m)|^{3/2}\dxs\\\nonumber
&\lesssim \int_0^t\int_{\mathcal B^+_1}|\overline{\bfv}_m|^{3/2}\dxs+ \int_0^t\int_{\mathcal B^+_1}|\zeta\nabla\overline{\bfv}_m|^{3/2}\dxs\\
&\lesssim 1+\kappa \int_0^t\int_{\mathcal B^+_1}\zeta^2|\nabla\overline{\bfv}_m|^{2}\dxs
\label{eq:Im1}
\end{align}
using also \eqref{I6} in the last step. Moreover, it holds by \eqref{eq:Psim2} and \eqref{lem:9.4.1}\footnote{Note that the assumption $p\geq 3/2$ would have been sufficient here, see \cite[Proposition 3.5.3.]{MaSh} for inclusions of the spaces $\mathcal M^{s,p}$.}
\begin{align}\label{eq:Im2}\begin{aligned}
(I)_m^2&
\lesssim \sum_{k=1}^3 \int_0^t\|(\zeta\overline{\bfv}_m)\circ\bfPsi_m\|_{W^{1,3/2}(\bfPhi_m(\mathcal B^+_1))}^{3/2}\|\bfPsi_m^k\|_{\mathcal M^{2,3/2}(\bfPhi_m(\mathcal B^+_1))}^{3/2}\ds\\
&
\lesssim \sum_{k=1}^3 \int_0^t|\|\zeta\overline{\bfv}_m\|_{W^{1,3/2}(\mathcal B_1^+)}^{3/2}\ds\\
&\lesssim \int_0^t\int_{\mathcal B^+_1}|\overline{\bfv}_m|^{3/2}\dxs+ \int_0^t\int_{\mathcal B^+_1}|\zeta\nabla\overline{\bfv}_m|^{3/2}\dxs\\
&\lesssim 1+\kappa \int_0^t\int_{\mathcal B^+_1}\zeta^2|\nabla\overline{\bfv}_m|^{2}\dxs.
\end{aligned}
\end{align}
Let us choose $\zeta$ such that $\zeta=1$ in $\mathcal Q_{3/4}$.
Finally, since $\bfA_{\varphi_m}$ is elliptic uniformly in $m$ and $J_{\varphi_m}$ strictly positive by \eqref{eq:detJ}, we conclude
\begin{align}\label{eq:regvm}
\overline{\bfv}_m\in L^\infty\big(\mathcal I_{3/4};L^2\big(\mathcal B^+_{3/4}\big)\big)\cap L^2\big(\mathcal I_{3/4};W^{1,2}\big(\mathcal B_{3/4}^+\big)\big)
\end{align}
uniformly in $m$ choosing $\kappa$ small enough.
By Sobolev's inequality we have for $\bfphi\in C^\infty_c(\mathcal Q_{3/4}^+)$
\begin{align*}
\int_{\mathcal Q^+_1}& \partial_t \overline{\bfv}_m\cdot \bfphi\dz\ds\leq \,c\,\Big(\|\overline{\bfv}_m\|_{L^2_z}\|\nabla\overline{\bfv}_m\|_{L^2_z}+\|\overline{\mathfrak{q}}_m\|_{L^{3/2}_z}+\|\overline\bfg_m\|_{L^2_z}\Big)\|\bfphi\|_{W^{2,2}_z}.
\end{align*}
This shows by \eqref{I4} and \eqref{eq:regvm} the boundedness of 
\begin{align}\label{eq:dtvm}
\partial_t \overline{\bfv}_m\in L^{2}(\mathcal I_{3/4};W^{-2,2}(\mathcal B^+_{3/4})).
\end{align}
After passing to suitable subsequences we obtain
\begin{align}
\label{31}\partial_t \overline{\bfv}_m&\rightharpoonup \partial_t\overline{\bfv}\quad\text{in}\quad L^{2}\big(\mathcal I_{3/4};W^{-2,2} 
\big(\mathcal B^+_{3/4}\big)\big).
\end{align}
Now, \eqref{eq:regvm} and \eqref{eq:dtvm} imply
\begin{align}\label{eq:nablavmcomp}
\overline\bfv_m&\rightarrow \overline\bfv\quad\text{in}\quad L^3\big(\mathcal I_{3/4};L^{3}\big(\mathcal B_{3/4}^+\big)\big)
\end{align}
by the Aubin-Lions compactness theorem.
Recalling the definitions of $\bfA_{\varphi_m}$ and $\bfB_{\varphi_m}$ from \eqref{eq:AB} we are able to pass to the limit in \eqref{eq:eq} using the convergences \eqref{eq:Phim2}, \eqref{eq:Psim2}, \eqref{I8}, \eqref{I8'} \eqref{eq:regvm} and \eqref{31}. We obtain (in the sense of distributions on $\mathcal Q^+_{3/4}$)
\begin{align}\label{eq:eqlimit}
\begin{aligned}
J_{\varphi}\partial_t\overline\bfv+\Div\big(\bfB_{\varphi}\overline{\mathfrak q}\big)-\Div\big(\bfA_{\varphi}\nabla\overline\bfv\big)&=J_\varphi\overline\bfg,\\
\bfB_{\varphi}^\top:\nabla{\overline\bfv}=0,\quad\overline\bfv|_{\mathcal B^+_{3/4}\cap\partial\mathbb H}&=0,
\end{aligned}
\end{align}
recalling that the comvective term disappears as $\lambda_m\rightarrow0$. Note that since $p>3$ relation
\eqref{eq:MSb} yields boundedness of $\bfPsi_m$ and $\bfPhi_m$ in $W^{2,p}$ such that  $\nabla\bfPsi_m$ and $\nabla\bfPhi_m$ are pre-compact.

We are now fully prepared to lead
\eqref{I7} to a contradiction applying the regularity theory developed in Section \ref{sec:pert} to \eqref{eq:eq} and \eqref{eq:eqlimit}. We infer from Lemma \ref{lem:32} applied to \eqref{eq:eqlimit}, using also \eqref{eq:varphim2}, that
\begin{align*}
\|\nabla^2\overline\bfv\|_{L^{5/3}(\mathcal I_{1/2};L^p(\mathcal B^+_{1/2}))}&+\|\partial_t\overline\bfv\|_{L^{5/3}(\mathcal I_{1/2};L^p(\mathcal B^+_{1/2}))}\\&\lesssim 
\|\nabla\overline\bfv\|_{L^{5/3}(\mathcal I_{3/4};L^{5/3}(\mathcal B^+_{3/4}))}+\|\overline{\mathfrak q}\|_{L^{5/3}(\mathcal I_{3/4};L^{5/3}(\mathcal B^+_{3/4}))}+\|\overline\bfg\|_{L^{5/3}(\mathcal I_{3/4};L^p(\mathcal B^+_{3/4}))},
\end{align*}
where the right-hand side is finite on account of \eqref{I8}, \eqref{I8''} and \eqref{eq:regvm}. We deduce from parabolic embeddings that $\overline\bfv$ belongs to the class $C^{\beta/2,\beta}_{\mathrm{para}}(\overline{\mathcal Q}^+_{3/4})$, where $\beta=4/5-3/p>\alpha$. Note that $\beta>0$ for
$p>\frac{15}{4}$. 
 This, \eqref{eq:nablavmcomp} and $\overline\bfv|_{\mathcal B^+_{3/4}\cap\partial\mathbb H}=0$ prove
\begin{align}\label{eq:1809}
\lim_{m\rightarrow\infty}\dashint_{\mathcal Q^+_\tau} |\overline\bfv_m|^3 \dz\ds=\dashint_{\mathcal Q^+_\tau} |\overline\bfv|^3 \dz\ds
\leq C_{\overline\bfv} \tau^{3\alpha}
\end{align}
for all $\tau<\frac{1}{2}$ for some constant $C_{\overline\bfv}>0$.
Let us consider the unique solution $(\tilde\bfv_m,\tilde{\mathfrak{q}}_m)$ to
\begin{align}\label{eq:eq'}
\begin{aligned}
J_{\varphi_m}\partial_t\tilde\bfv_m+\Div\big(\bfB_{\varphi_m}\tilde{\mathfrak{q}}_m\big)-\Div\big(\bfA_{\varphi_m}\nabla\tilde\bfv_m\big)&=-\lambda_m(\bfB_{\varphi_m}\nabla\overline\bfv_m)\overline\bfv_m,\\
\bfB_{\varphi_m}^\top:\nabla\tilde{\bfv}_m=0,\quad\tilde\bfv_{m}|_{\mathcal B^+_{3/4}\cap\partial\mathbb H}&=0,\quad\overline{\bfv}_m(-3/4,\cdot)=0,
\end{aligned}
\end{align}
where the first equation is understood in the sense of distributions on $\mathcal Q^+_{3/4}$.
By the regularity theory for the perturbed Stokes system established in Lemma \ref{lem:31} (which applies on account of \eqref{eq:varphim2}) we have in $\mathcal Q^+_{3/4}$
\begin{align}\label{eq:1809b}
\begin{aligned}
\|\nabla^2\tilde\bfv_m\|_{L^{5/3}_tL^{{15/14}}_x}&+\|\nabla\tilde{\mathfrak{q}}_m\|_{L^{5/3}_tL^{{15/14}}_x}\\&\lesssim\lambda_m\|(\bfB_{\varphi_m}\nabla\overline\bfv_m)\overline\bfv_m\|_{L^{5/3}_tL^{{15/14}}_x}\longrightarrow0,\quad m\rightarrow\infty,
\end{aligned}
\end{align}
using \eqref{eq:varphim} and
\eqref{eq:regvm}. Now we consider the difference
$\tilde\bfv^\ast_m=\tilde\bfv_m-\overline\bfv_m$, $\tilde{\mathfrak{q}}_m^\ast=\tilde{\mathfrak{q}}_m-\overline{\mathfrak{q}}_m$ which solves a Stokes system with right-hand side $J_{\varphi_m}\overline\bfg_m$ in $\mathcal Q_{3/4}^+$
and hence satisfies
\begin{align*}
\|\nabla\tilde{\mathfrak{q}}^\ast_m\|_{L^{5/3}_tL^{p}_x}\lesssim\|\overline\bfg_m\|_{L^{5/3}_tL^{p}_x}+\|\nabla\tilde\bfv^\ast_m\|_{L^{5/3}_tL^{{15/14}}_x}+\|\tilde{\mathfrak{q}}^\ast_m-(\tilde{\mathfrak{q}}_m^\ast)_{\mathcal B^+_{3/4}}\|_{L^{5/3}L^{{15/14}}}
\end{align*}
by Lemma \ref{lem:32} using also \eqref{eq:varphim2}. Here the norm on the left-hand side is taken over $\mathcal Q_{1/2}^+$ and the one on the right-hand side over $\mathcal Q_{3/4}^+$.
The first term on the right-hand side is bounded by \eqref{I6} and the second one by \eqref{eq:regvm} and \eqref{eq:1809b} (recall that $\tilde\bfv_m$ has zero boundary conditions). For the third one we have
\begin{align*}
\|\tilde{\mathfrak{q}}^\ast_m-(\tilde{\mathfrak{q}}^\ast_m)_{\mathcal B^+_{3/4}}\|_{L^{5/3}_xL^{{15/14}}_x}&\lesssim\|\overline{\mathfrak{q}}_m-(\overline{\mathfrak{q}}_m)_{\mathcal B^+_{3/4}}\|_{L^{5/3}_tL^{{15/14}}_x}+\|\tilde{\mathfrak{q}}_m-(\tilde{\mathfrak{q}}_m)_{\mathcal B^+_{3/4}}\|_{L^{5/3}_tL^{{15/14}}_x}\\
&\lesssim\|\overline{\mathfrak{q}}_m\|_{L^{5/3}_tL^{{15/14}}_x}+\|\nabla\tilde{\mathfrak{q}}_m\|_{L^{5/3}_tL^{{15/14}}_x},
\end{align*}
which is bounded by \eqref{I8} and \eqref{eq:1809b}. We conclude, that
$
\|\nabla\tilde{\mathfrak{q}}^\ast_m\|_{L^{5/3}_tL^{p}_x}$ (with norm taken over $\mathcal Q_{1/2}^+$) is bounded
which yields
\begin{align*}
\tau^{3}\bigg(\dashint_{\mathcal I_\tau}\dashint_{\mathcal B^+_\tau}  |\tilde{\mathfrak{q}}^\ast_m  - (\tilde{\mathfrak q}_{m}^\ast)_{\mathcal B^+_{\tau}} |^{5/3} \dz\ds\bigg)^{\frac{9}{5}}
&\lesssim\tau^{-3}\bigg(\int_{-\tau}^{\tau}\int_{\mathcal B^+_\tau}  |\nabla\tilde{\mathfrak{q}}^\ast_m |^{5/3} \dz\ds\bigg)^{\frac{9}{5}}\\
&\leq\tau^{-3}
\bigg(\int_{-\tau}^{\tau}\tau^{3-\frac{5}{p}}\bigg(\int_{\mathcal B^+_\tau}  |\nabla\tilde{\mathfrak{q}}^\ast_m |^{p} \dz\bigg)^{\frac{5}{3p}}\ds\bigg)^{\frac{9}{5}}\\
&\lesssim\tau^{2\alpha},
\end{align*}
where $\alpha=3\big(\frac{2}{5}-\frac{3}{2p}\big)$.
Combining this with \eqref{eq:1809b} and Sobolev's embedding $W^{1,15/14}_x\hookrightarrow L^{5/3}_x$ shows
\begin{align*}
\limsup_m \tau^3\bigg(&\dashint_{\mathcal I_\tau}\dashint_{\mathcal Q^+_\tau}  |\overline{\mathfrak{q}}_m  - (\overline{\mathfrak q}_{m})_{\mathcal B^+_{\tau}} |^{5/3} \dz\ds\bigg)^{\frac{9}{5}}\\&=
\limsup_m \tau^3\bigg(\dashint_{\mathcal I_\tau}\dashint_{\mathcal B_\tau^+}  |\tilde{\mathfrak{q}}^\ast_m  - (\tilde{\mathfrak q}^\ast_{m})_{\mathcal B^+_{\tau}} |^{5/3} \dz\ds\bigg)^{\frac{9}{5}}\leq \,C_{\overline{\mathfrak q}}\tau^{2\alpha}
\end{align*}
for some constant $C_{\overline{\mathfrak q}}$.
This, together with \eqref{eq:1809}, contradicts \eqref{I7} if we choose  $C_\ast>C_{\overline\bfv}+C_{\overline{\mathfrak q}}$. 
\end{proof}

\begin{proof}[Proof of Theorem \ref{thm:main}]
By assumption we have
\begin{align*}
r^{-2}\int_{t_0-r^2}^{t_0+r^2}\int_{\Omega\cap \mathcal B_r(x_0)}|\bfu|^3\dxt+\bigg(r^{-5/3}\int_{t_0-r^2}^{t_0+r^2}\int_{\Omega\cap \mathcal B_r(x_0)}|\pi|^{5/3}\dx\dt\bigg)^{\frac{9}{5}}<\varepsilon_0.
\end{align*}
After rotation and translation of the coordinate system we can assume that $(t_0,x_0)=(0,0)$. With the mapping $\bfPhi$ from \eqref{eq:Phi} we turn to the functions $\overline \pi=\pi\circ\bfPhi$ and $\overline{\bfu}=\bfu\circ\bfPhi$ which solve the perturbed system \eqref{momref} in $I\times \mathcal B^+_{\lambda\mathfrak R}$. Here $\mathfrak R$ is given in Remark \ref{rem:cover} and $\lambda$ is chosen such that $\bfPhi(\mathcal B^+_{\lambda\mathfrak R})\subset\Omega\cap B_{\mathfrak R}$ (recall that $\bfPhi$ is Lipschitz). We obtain with $R:=\lambda r$
\begin{align*}
R^3\mathscr E_{R}^{t_0,x_0}&(\overline \bfu,\overline\pi)\lesssim R^{-2}\int_{t_0-R^2}^{t_0+R^2}\int_{ \mathcal B^+_{R}}|\overline\bfu|^3\dxt+R^{-3}\bigg(\int_{t_0-R^2}^{t_0+R^2}\int_{\mathcal B_R^+}|\overline\pi|^{5/3}\dx\dt\bigg)^{\frac{9}{5}}\\&\lesssim
r^{-2}\int_{t_0-r^2}^{t_0+r^2}\int_{\Omega\cap \mathcal B_r(x_0)}|\bfu|^3\dxt+\bigg(r^{-5/3}\int_{t_0-r^2}^{t_0+r^2}\int_{\Omega\cap \mathcal B_r(x_0)}|\pi|^{5/3}\dx\dt\bigg)^{\frac{9}{5}}
  <\varepsilon_{0}.
\end{align*}
Furthermore, it holds
\begin{align*}
R^{9}\bigg(\dashint_{t_0-R^2}^{t_0+R^2}\dashint_{\mathcal B_R^+}|\overline\bff|^{p}\dx\dt\bigg)^{\frac{3}{p}}=\bigg(R^{3p-5}\int_{t_0-R^2}^{t_0+R^2}\int_{\mathcal B_R^+}|\overline\bff|^{p}\dx\dt\bigg)^{\frac{3}{p}}<\varepsilon_0
\end{align*}
provided we choose $R$ small enough. If $\varepsilon_0$ is small enough we obtain
\begin{equation*}
R^3\mathscr E_{R}^{t_0,x_0}(\overline\bfu,\overline\pi) + 
R^9\bigg(\ddashint_{\mathcal Q^+_R}  |\overline\bff|^{p} \dy\ds\bigg)^{\frac{3}{p}}
\leq\varepsilon,
\end{equation*}
such that a scaled version of Lemma \ref{Lemma} applies (note that $\varphi$ must be scaled to $\frac{1}{R}\varphi(R\cdot)$ such that its multiplier norm and Lipschitz constant are still bounded by $\delta$ as long as $R\leq1$). We conclude
that
\begin{equation*}
R^3\mathscr E^{t_0,x_0}_{\tau R} (\overline\bfu,\overline\pi) \le C_\ast \tau^{2\alpha} \bigg(R^3\mathscr E_R^{t_0,x_0} (\overline\bfu,\overline\pi) +R^9\bigg(\ddashint_{\mathcal Q^+_R}  |\overline\bff|^{p} \dy\ds\bigg)^{\frac{3}{p}}\bigg).
\end{equation*}
It is standard to iterate this inequality (see, e.g., \cite[Prop. 2.5]{EsSeSv}), which yields
\begin{align*}
\mathscr E_{\tau^k R}^{t_0,x_0}(\overline \bfu,\overline\pi)\lesssim \tau^{2\alpha k}.
\end{align*}
We trivially obtain
\begin{align}
\widetilde{\mathscr E}_{\tau^k R}^{t_0,x_0}(\overline\bfu,\overline\pi)\lesssim \tau^{2\alpha k},\label{decay1}
\end{align}
where the access $\widetilde{\mathscr E}$ is given by
\begin{align*}
\widetilde{\mathscr E}_{r}^{t_0,x_0} (\overline\bfu,\overline\pi)  &:=  \dashint_{\mathcal Q_{r}(t_0,x_0)\cap (I\times\Omega)} |\overline\bfu-(\overline\bfu)_{\mathcal Q_{r}(t_0,x_0)\cap (I\times\Omega)}|^3 \dy\ds\\&+ r^3\bigg(\dashint_{\mathcal I_r(t_0)}\dashint_{\mathcal B^+_{r}(x_0)}  |\overline\pi - (\overline\pi)_{\mathcal B_r(x_0)\cap\Omega}|^{5/3} \dy\ds\bigg)^{\frac{9}{5}}.
\end{align*}
The decay estimate \eqref{decay1} holds in the centre point $(t_0,x_0)=(0,0)$ of $\mathcal Q_1^+$ but also for points $(t_0,x_0)$ on $I\times \partial\Omega$ which are sufficiently close. We claim that it continues to hold in the interior of $\Omega$. In fact, we can prove a version of Lemma \ref{Lemma} for interior points for which $\mathcal Q_R(t_0,x_0)\subset \mathcal Q^+$. The main difference is that one has to replace in the proof Lemmas \ref{lem:31} and \ref{lem:32} by their corresponding interior versions Lemmas \ref{lem:31int} and \ref{lem:32int}. Combining the interior and the boundary version yields
\begin{align*}
\widetilde{\mathscr E}_{\tau^k R}^{t_0,x_0}(\bfu,\pi)\lesssim r^{2\alpha}.
\end{align*}
This proves that $\overline\bfu\in C^{0,\alpha}$ in a neighborhood of $(0,0)$. Changing coordinates (that is, recalling that $\bfu=\overline\bfu\circ\bfPsi$ where $\bfPsi$ is Lipschitz, cf. \eqref{eq:detJ}) and changing the coordinate system as the case may be, proves $\bfu\in C^{0,\alpha}(\overline{\mathcal U}(t_0,x_0))$ for some  neighborhood $\mathcal U(t_0,x_0)$ of $(t_0,x_0)$.
\end{proof}

%

\section{Concluding remarks \& outlook}

\subsection{The size of the singular set}
\label{sec:setsize}
We comment in this section on the fact that our estimate  on the size of the singular set
in Theorem \ref{thm:main'} (dimension $\leq 5/3$) is weaker 
than that from \cite{SeShSo} for regular boundaries. 
We introduce the dissipation functional
\begin{align*}
\mathscr D_r(\overline\bfu)=\dashint_{\mathcal Q_r^+}|\nabla\overline\bfu|^2\dxt.
\end{align*}
The following implication is proved in \cite[proof of Theorem 5.1]{SeShSo}:
\begin{align}\label{eq:DE}
\sup_{r<1}\frac{1}{r}\int_{\mathcal Q_r^+}|\nabla\overline\bfu|^2\dxt<\varepsilon_{\mathscr D}\quad\Rightarrow \quad \liminf_{r\rightarrow0}r^3\mathscr E_r(\overline\bfu,\overline\pi)<\varepsilon,
\end{align}
where $\varepsilon_{\mathscr D}$ is a small number. The size of the set where the first condition is violated is much smaller than the size of the set where the second one is violated (the parabolic Hausdorff-dimensions are 1 and 5/3). Let us explain the strategy to prove the implication \eqref{eq:DE}.
Transforming the local energy inequality to the flat geometry (as we did in \eqref{energylocal}) and diving by $r$, they proof (in our notation)
\begin{align}\label{eq:2101}
\sup_{\mathcal I_{3r/4}}\frac{1}{r}\int_{\mathcal B_{3r/4}^+}|\overline\bfu|^2\dx+\frac{1}{r}\int_{\mathcal Q_{3r/4}^+}|\nabla\overline\bfu|^2\dxt&\lesssim r^3\mathscr E_r(\overline\bfu,\overline\pi)+\big(r^3\mathscr E_r(\overline\bfu,\overline\pi)\big)^{2/3}.
 \end{align}
Arguing as in \eqref{eq:Im1} and \eqref{eq:Im2} and appreciating the correct scaling (that is, we have $|\nabla\zeta|\lesssim r^{-1}$, $|\partial_t\zeta|\lesssim r^{-2}$ and $|\nabla^2\zeta|\lesssim r^{-2}$), the additional terms
\begin{align*}
 \frac{1}{r^{7/2}}\int_{\mathcal I_r}\int_{\mathcal B^+_r}|\overline\bfu|^{3/2}\dxs\quad\text{and}\quad \frac{1}{r^2}\int_{\mathcal I_r}\int_{\mathcal B^+_r}|\zeta\nabla\overline\bfu|^{3/2}\dxs
\end{align*}
appear in our situation.
The first term can be estimated by 
\begin{align*}
 \frac{1}{r^{2}}\int_{\mathcal I_r}\int_{\mathcal B^+_r}|\overline\bfu|^{3}\dxs+ \frac{1}{r^{5}}\int_{\mathcal I_r}\int_{\mathcal B^+_r}\dxs\lesssim r^3\mathscr E_r(\overline\bfu,\overline\pi)+1
\end{align*}
and, similarly, we have the upper bound
\begin{align*}
 \frac{\kappa}{r}\int_{\mathcal I_r}\int_{\mathcal B^+_r}|\zeta\nabla\overline\bfu|^{2}\dxs+ c_\kappa
\end{align*}
with $\kappa>0$ arbitrary for the second one.
Here the remaining integral can be absorbed in the left-hand side of \eqref{eq:2101} if $\kappa$ is sufficiently small. In conclusion, we obtain an additional (additive) constant on the right-hand side of \eqref{eq:2101}. We could allow a quantity which disappears in the limit $r\rightarrow0$, but a nontrivial constant destroys the argument for the proof of \eqref{eq:DE}.

\subsection{Optimality of the assumptions on the boundary}
\label{sec:opt}
We are able to significantly relax the assumptions regarding the regularity of the boundary made in \cite{SeShSo} ($C^2$ versus $W^{2-1/p,p}$ for some $p>\frac{15}{4}$). However, it is unclear if our assumptions in Theorems \ref{thm:main} and \ref{thm:main'} are optimal or if they can still be weakened. The assumptions on the coefficients for the perturbed Stokes system in Lemma \ref{lem:31} are certainly optimal as the theory of Sobolev multipliers already yields optimal results for the Laplace equation, see \cite[Chapter 14]{MaSh}.
There are various stages in the proof of the blow up lemma (see Lemma \ref{Lemma}), which require restrictions on the regularity of $\bfPhi$ (and hence that of $\varphi$):
\begin{itemize}
\item In order to control the local energy inequality in the flat geometry second derivatives of $\bfPhi_m$ appear (see $(I)_m^2$ in the proof of \eqref{eq:nablavmcomp}). We need $\bfPhi_m\in \mathcal M^{2,3/2}(\delta)$ (and hence $\varphi_m\in \mathcal M^{4/3,3/2}(\delta)$) to estimate the critical term.
We do not expect that it is possible to relax this.
\item We need the limit function $\overline\bfv$ - solution to some perturbed Stokes system - to be H\"older-continuous.
By parabolic embeddings this follows from an estimate in
\begin{align*}
L^{5/3}_tW^{2,p}_x\cap W^{1,5/3}_tL^p_x,
\end{align*}
which requires $\bfPhi_m\in \mathcal M^{2,p}(\delta)$ (and hence $\varphi_m\in \mathcal M^{2-1/p,p}(\delta)$) provided $p>\frac{15}{4}$.
However, we believe that H\"older-continuity of $\overline\bfv$ can be proved under much weaker assumptions. Unfortunately, we were unable to trace a suitable reference. Of course, it will be absolutely necessary to have continuity of solutions of the perturbed Stokes system. The latter is used as a local comparison system and continuity of the solution to the perturbed Navier--Stokes system cannot  be expected otherwise.
\item In order to control the decay of the pressure to arrive at the contradiction in the blow up lemma (see Lemma \ref{Lemma})
we need an estimate for $\nabla\tilde{\mathfrak q}_m$ in $L^{5/3}_t L^p_x$ with some $p>\frac{15}{4}$. For this, the same assumptions on $\bfPhi_m$ as in the previous bullet point are needed. Here,
we expect that it will be very difficult to relax this. Note in particular, that the Stokes system does not allow for estimates on the pressure in spaces with differentiability less than 1 (as in this case the time derivative only exists as a distribution on the solenoidal test-functions). This estimate also motivated the choice of our excess functional for which the time-integrability of the pressure is chosen as large as possible, while it should still coincide with the space integrability (in fact, one could allow a smaller space integrability, but this does not seem to improve the estimate for the size of the singular set).
The more customary choice
\begin{align*}
{\mathscr E}_r^{t_0,x_0} (\overline\bfu,\overline\pi)  &:=  \dashint_{\mathcal Q^+_{r}(t_0,x_0)} |\overline\bfu|^3 \dy\ds+ r^3\dashint_{\mathcal Q^+_{r}(t_0,x_0)}  |\overline\pi - (\overline\pi)_{\mathcal Q^+_r(x_0)}|^{\frac{3}{2}} \dy\ds,
\end{align*}
introduced initially in \cite{L},
requires the more restrictive assumption $p>\frac{9}{2}$.
The original excess functional from \cite{CKN}, which is based on
the function space $L_t^{5/4}L^{5/3}_x$ for the pressure, leads to the condition $p>\frac{15}{2}$ instead.
In general, for $\pi\in L_t^{r_\ast}L^{s_\ast}_x$ one has the condition
$p>\tfrac{3r_\ast}{2 r_\ast-2}$.
\end{itemize}

\subsection{Fluid-structure interaction}
\label{sec:fsi}
In a typical problem from fluid-structure interaction an elastic structure is located at a non-trivial part $\Gamma\subset \partial\Omega$, where $\Omega$ plays the role of a reference geometry.  The deformation of the structure is described by a function $\eta:I\times \Gamma\rightarrow\R$ and the domain $\Omega$ is deformed to 
$\Omega_{\eta(t)}$ defined through its boundary
\begin{align*}
\partial\Omega_{\eta(t)}=\{y+\eta(t,y)\nu(y):\,y\in\partial\Omega\},
\end{align*}
where $\nu$ is the outer unit normal at $\partial\Omega$. The Navier--Stokes equations are posed in the moving space-time cylinder given by
\begin{align*}
I\times\Omega_\eta:=\bigcup_{t\in I}\{t\}\times\Omega_{\eta(t)}
\end{align*}
with some abuse of notation.
The displacement $\eta$ is the solution to a hyperbolic equation, a proto-typical example is
\begin{align}\label{eq:eta}
\varrho_s\partial_t^2\eta+\beta\Delta^2\eta=\bfF\quad\text{in}\quad I\times \Gamma
\end{align}
with $\varrho_s,\beta>0$ describing a linearised Koiter-type shell.
The function $\bfF$ on the right-hand side describes the response of the structure to the surface forces of the fluid imposed by the Cauchy stress. The existence of a weak solution to the coupled system has been shown in \cite{LeRu}. Solutions belong to the energy space which means for the structure that
\begin{align*}
\eta\in W^{1,\infty}(I;L^2(\Gamma))\cap L^\infty(I;W^{2,2}(\Gamma)).
\end{align*}
The second function space (and Sobolev's embedding in two dimensions) indicates that the boundary of $\Omega_{\eta(t)}$ is not even Lipschitz. With only this information at hand we do not expect that a theory similar to that in the present paper is reachable. It has, however, been shown very recently in \cite{MuSc} that solutions satisfy additionally
\begin{align}\label{eq:1401b}
\eta\in W^{1,2}(I;W^{\theta,2}(\Gamma))\cap L^2(I;W^{2+\theta,2}(\Gamma)).
\end{align}
for all $\theta<\frac{1}{2}$. Due to the compact embeddings
\begin{align*}
W^{2+\theta,2}(\Gamma)\hookrightarrow\hookrightarrow W^{1,\infty}(\Gamma),\quad W^{2+\theta,2}(\Gamma)\hookrightarrow\hookrightarrow W^{2-1/p,p}(\Gamma),
\end{align*}
from some $p>\tfrac{15}{4}$ we obtain $\mathrm{Lip}(\partial\Omega_{\eta(t)})\leq\delta$ and $\partial\Omega_{\eta(t)}\in\mathcal M^{2-1/p,p}(\R^2)(\delta)$ as required in Theorems \ref{thm:main} and \ref{thm:main'}.
Unfortunately, this smallness is not uniformly in time. 
While we expect that methods similarly to those developed here can also be applied to study Navier--Stokes equations in moving domains (this is already highly non-trivial and requires some additional research), it is not clear if the regularity from \eqref{eq:1401b} will be sufficient. Some further analysis is required.

\section*{Compliance with Ethical Standards}\label{conflicts}
\smallskip
\par\noindent
{\bf Conflict of Interest}. The author declares that he has no conflict of interest.

\smallskip
\par\noindent
{\bf Data Availability}. Data sharing is not applicable to this article as no datasets were generated or analysed
during the current study.

\end{document}